\documentclass[english,structabstract]{amsart}

\usepackage{epsfig,verbatim,amssymb,amsmath,amsfonts,graphicx,color}
\usepackage{cleveref}
\usepackage{setspace}

\newtheorem{theorem}{Theorem}[section]

\newtheorem{lemma}[theorem]{Lemma}
\newtheorem{prop}[theorem]{Proposition}
\newtheorem{cor}[theorem]{Corollary}

\newtheorem{remark}[theorem]{Remark}

\newtheorem{defi}[theorem]{Definition}
\newtheorem{problem}[theorem]{Problem}

\numberwithin{equation}{section}

\newcommand{\R}{\mathbb{R}}

\newcommand{\Z}{\mathbb{Z}}
\newcommand{\N}{\mathbb{N}}

\renewcommand{\tilde}{\widetilde}

\newcommand{\om}{\omega}
\newcommand{\Om}{\Omega}

\newcommand{\PP}{\mathbb{P}}
\newcommand{\plo}{\mathbb{P}_{LO}}
\newcommand{\E}{\mathbb{E}}

\newcommand{\cookieenv}{\omega}
\newcommand{\QuenchedCookies}{\mathbb{P}_\om}
\newcommand{\QuenchedWalks}{\mathbb{P}_{\om,0}}
\newcommand{\QuenchedTrees}{\mathbb{P}_\om^1}

\newcommand{\AnnealedCookies}{\mathbb{P}}
\newcommand{\AnnealedWalks}{\mathbb{P}_0}

\renewcommand{\epsilon}{\varepsilon}

\newcounter{constante}
\newcommand{\con}[1]{
\immediate\write 1{\noexpand\newlabel{#1}{{\theconstante}{\theconstante}}}
                  c_{\theconstante}
                    \stepcounter{constante}
                   }

\newcommand{\WLOG}{without loss of generality}

\author{Gideon Amir$^\sharp$ \and  Tal Orenshtein$^\dag$}
\thanks{$\sharp$ Bar Ilan University {\tt gidi.amir@gmail.com } \\
$\dag$ Weizmann Institute of Science and Technische Universit\"at M\"unchen {\tt  tal.orenshtein@weizmann.ac.il}}


\title[Excited Mob] {Excited Mob}

\begin{document}

\setcounter{page}{1}

\maketitle

\begin{abstract}
  We show that for an i.i.d.\ bounded and elliptic cookie environment, a one dimensional excited random walk on the $k$-time leftover environment is transient to the right if and only if $\delta>k+1$ and has positive speed if and only if $\delta>k+2$, where $\delta$ is the expected drift per site. A slightly different definition of leftover environments then gives, to the best of our knowledge, the first example of a stationary ergodic environment with positive speed that does not follow by trivial comparison to an i.i.d.\ environment.
  In another formulation, we show that on such environments an excited mob of $k$ walkers is transient to the right if and only if $\delta>k$ and moves with positive speed if and only if $\delta>k+1$. We also prove a 0-1 law for directional transience and law of large numbers for leftover environments of stationary ergodic and elliptic cookie environments.
\end{abstract}

\pagestyle{myheadings}
\markboth{Excited Mob}{G. Amir and T. Orenshtein}
\section{Introduction}
\subsection{Basic model and notations}
 Excited random walk on $\Z^d$, $d\ge 1$, was introduced by Itai~Benjamini and David~B.~Wilson in 2003 \cite{benjamini2003excited}. The model in dimension $d=1$ was generalized by Martin~P.~W.~Zerner \cite{zerner2005multi}. It was studied extensively in recent years by numerous authors, and an almost up to date account may be found in the recent survey of Kosygina and Zerner \cite{kosygina2012excited}.

The model is defined as follows.
Let $\Omega=[0,1]^{\Z\times\N}$ and endow this space with the Borel $\sigma$-algebra generated by the Tychonoff product topology, where the space $[0,1]$ has the standard real topology. We call $\Om$ the space of cookie environments, where each $\om\in\Om$ is a cookie environment. $\om(x,n)\in[0,1]$ is called the $n$-th cookie in the location $x$.

Given a cookie environment $\om$ and an initial position $x\in\Z$ the excited random walk $X=(X_n)_{n\ge0}$ driven by $\om$ is given by:
\begin{eqnarray*}
\PP_{\om,x}(X_0=x)=1, \\
\PP_{\om,x}(X_n=X_{n-1}+1\ |\ X_0,X_1,\ldots,X_{n-1}) &=& \om(X_{n-1},\#\{k\leq n-1: X_k=X_{n-1}\}), \\
\PP_{\om,x}(X_n=X_{n-1}-1\ |\ X_0,X_1,\ldots,X_{n-1}) &=& 1 - \PP_\om(X_n=X_{n-1}+1\ |\ X_0,X_1,\ldots,X_{n-1}).
\end{eqnarray*}

The probability measure $\PP_{\om,x}$ is called the \emph{quenched} measure on the excited random walks started from $x$. Given a probability measure $P$ on the space $\Om$ of cookie environments, with a corresponding expectation operator $E$, we define the \emph{annealed} (also called \emph{averaged}) measure $\PP_x$ to be the $P$-average of the quenched measure:
\[
\AnnealedCookies_x[\cdot]=E[\PP_{\om,x}(\cdot)].
\]

The following two assumptions on the measure $P$ on cookie environments are standard.
We adopt the notations of \cite{kosygina2012excited}.

\begin{equation*}
\mbox{The family $(\omega(x,\cdot))_{x\in\Z}$ of cookie stacks is i.i.d.\ under
$P$}
\tag{IID}
\end{equation*}
and
\begin{equation*}
\mbox{\begin{tabular}{l} The family $(\omega(x,\cdot))_{x\in\Z}$ is stationary
 and ergodic\footnotetext{i.e.\ every Borel measurable $A\subset\Om$ which is invariant under left or right shifts on $\Z$ satisfies $P[A]\in\{0,1\}$.}
  under $P$
  \\ with respect to the shift on $\Z$.
\end{tabular}}\tag{SE}
\end{equation*}

Define also the following properties of a cookie environment $\om$: ellipticity, non-degeneracy, and positivity.

\begin{equation*}
  \mbox{\begin{tabular}{l} $\om(x,n)\in(0,1)$ for all $x\in\Z$ and $n\in\N$.\end{tabular}}\tag{ELL}
\end{equation*}
\begin{equation*}
  \mbox{\begin{tabular}{l}
  $\sum_{i=1}^\infty\om(x,i)=\infty$ and $\sum_{i=1}^\infty(1-\om(x,i))=\infty$ for all $x\in Z$.\end{tabular}}\tag{ND}
\end{equation*}
\begin{equation*}
  \mbox{\begin{tabular}{l} $\om(x,n)\ge\frac{1}{2}$ for all $x\in\Z$ and $n\in\N$.\end{tabular}}\tag{POS}
\end{equation*}

Say that a probability measure $P$ on $\Om$ satisfies (ELL), (ND) or (POS), respectively if $P$-a.s.\ $\om$ satisfies it.
The Non-degeneracy condition (ND) implies that for almost every environment $\om$, the walk is either transient or a.s.\ visits all vertices infinitely often.
(Without it other behaviors, such as being stuck in a finite interval, are possible).

Last, we define two additional properties on $P$: boundedness and weak ellipticity.
\begin{equation*}
  \mbox{\begin{tabular}{l} There is some deterministic $M$ such that $P$-a.s.\ \\
  $\om(x,n)=\frac{1}{2}$ for all $x\in\Z$ and $n>M$.\end{tabular}}\tag{BD}
\end{equation*}
\begin{equation*}
  \mbox{\begin{tabular}{l} For all $x\in \Z$ $P( \om(x,n)>0 \ \forall n\in\N)>0\ $ and $P(\om(x,n)<1 \ \forall n\in\N)>0\ $. \end{tabular}}\tag{WEL}
\end{equation*}
%\subsection{Previous work, leftover environments and main results}

\subsection{Motivation and main results}
Given a random walk on $\Z$ a fundamental question one is interested in is whether the walk is transient or recurrent, and, in case it is transient, does the walk obey a law of large numbers and does it escape to $\pm\infty$ with non-zero speed. In the case of random cookie environments, the best one can hope for is exact criterions for these properties in terms of the measure over environments, and when such criteria are lacking, to prove these properties hold with probability either zero or one.
We first address the question of recurrence vs. transience. Let $P$ be a probability measure over the space of cookie environments. Zerner \cite{zerner2005multi} gave an exact criterion in the case that $P$ is (SE) and (POS) in terms of the \emph{expected drift per site}
$$\delta(x,P)= E \left [ \sum_{i=1}^\infty(2\om(x,i)-1) \right ].$$
Whenever $P$ is stationary then $\delta(x,P)=\delta(P)$ is independent of $x\in\Z$. We will sometimes use $\delta=\delta(P)$ when there is no confusion, and note that $\delta$ is well-defined (possibly infinite) under (POS) or (BD). Zerner showed that under (SE) and (POS) the walk is transient to the right if and only if either $\delta>1$ or $P(\om(0,1)=1)=1$. It was later shown by Kosygina and Zerner \cite{kosygina2008positively} that under (BD),~(IID) and~(WEL) the same threshold holds.
For the case that the measure satisfies (SE) no such threshold is known, however under additional condition, (ELL), Kosygina and Zerner \cite{kosygina2012excited} showed that a 0-1 law for transience still holds, and in \cite{amir2013zero} it was shown that also a 0-1 law for directional transience holds in these settings.

When coming to the question of having positive speed, much less in known. The techniques in the literature deal with measures that are (IID), (BD) and (WEL). It was shown by Basdevant and Singh \cite{basdevant2008speed} and Kosygina and Zerner \cite{kosygina2008positively} that once again an exact criterion can be formulated in terms of $\delta$. The walk has positive speed if and only if $\delta > 2$. The (IID) structure of the environment is used extensively in these methods, and the (BD) condition is used to allow the usage of theorems regarding branching processes with migration, we will discuss this in Section \ref{sec:BPwM}.

In fact, it is not hard to show that in the case of (SE), ballisticity does not depend only on $\delta$. One can build an example of probability measures over cookie environments satisfying (SE), (POS), (BD) and (UEL) with $\delta$ as high as desired where the walk will a.s.\ have $0$ speed. Such an example was first constructed by Mountford, Pimentel and Valle in \cite{MPV06}. A simple sketch follows:
Take some positive integer valued random variable $T$ with $\E(T)<\infty$ and $\E(T^2)=\infty$, and create from it a stationary ergodic point process on $\Z$ with i.i.d.\ interval lengths distributed like $T$. Put ``walls" - infinite stacks of cookies with bias $1$ to the right - in places where there are points from the process, and no cookies anywhere else. It is not hard to see that on such an environment the excited walk will go to infinity a.s., but its speed will be $0$. As noted in \cite{kosygina2012excited}, the same argument works if we replace the infinite stacks with stack with arbitrary high bounded stacks, and reducing their strength from $1$ to $p>1/2$. For a more detailed example and other related topics the reader is referred to \cite{MPV06} and example $5.7$ in \cite{kosygina2012excited}.

Since it is not possible to generalize the results regarding positive speed in the (IID) case to general case of (SE), and since, to the best of our knowledge, all known examples of environments giving positive speed stem directly from the (IID) case, it is interesting to find natural families of (SE) measures which are not (IID) to which such results could be extended. Our main results regard the analysis of such a family of random environments which are interesting in their own right - the leftover environments.
Given a cookie environment on $\Z$, Zerner \cite{zerner2005multi} introduced the leftover cookie environment as the environment of cookies that were left over by the walker on the original environment (i.e.\ the cookies that have not been eaten by the walker throughout his movement). If the walker is transient, this is well-defined (see section \ref{sec:arrows}) while if the walker is recurrent we will simply say that all cookies were eaten, that is the leftover environment $\om$ in this case is given by $\om(x,n)=\frac{1}{2}$ for all $x\in\Z$ and $n\in\N$. When the walker is transient, this environment is random (even when the original environment is deterministic). One cannot hope that the leftover environment will inherit (SE) from the original environment (because of the significance of the starting position), and in particular the leftover environment is far from being (IID) even for very nice transient environments. Having said that, note that first, it was shown by Zerner \cite{zerner2005multi} that the leftover environment does inherit a weaker - ``directional" - form of (SE) from the original environment (see Section \ref{sec:LLN}), and second, one can also define a variant of the leftover environment which will be (SE) by taking the initial position of the original walker to $-\infty$ properly (see section \ref{sec:stationary})
%(the original proof was under the additional assumption (POS), but holds in greater generality),  for details.

 One can iterate the construction and consider the $k$-leftover environment by looking at the environment remaining after a walker walks on the $(k-1)$-leftover environment, $k> 1$. Denote by $X^{(1)}:=X$ the (standard) excited random walk, and sequentially, for $k > 1$, let $X^{(k)}$ be the excited walk on the $(k-1)$-leftover environment.

Our first result shows that under mild assumptions a generalization of the 0-1 law for directional transience \cite{amir2013zero} holds. In particular, we show that if the first walker $X^{(1)}$ is transient in one direction, then a.s.\ no walker $X^{(i)}$ will be transient in the other direction.

\begin{theorem}[0-1 law for directional transience]\label{thm:TransienceThreshold}
Assume that $P$ is a probability measure over the space of cookie environments satisfying either (SE), (ELL) and (ND) or (IID), (WEL) and (BD). Then there is some $R=R(P)\in\Z\cup\{-\infty,+\infty\}$ such that the following hold.
 \begin{itemize}
 \item $X^{(k)}$ is a.s.\ recurrent if $k > |R|$
    \item $X^{(k)}_n\to +\infty$ a.s.\ if $k \le |R|$ and $R>0$
     \item $X^{(k)}_n\to -\infty$ a.s.\ if $k \le |R|$ and $R<0$
    \end{itemize}
\end{theorem}

As mentioned in the beginning of this section, Zerner \cite{zerner2005multi} showed that under (POS) and a slightly weaker version of (SE), the walk $X^{(1)}$ is (right) transient if and only if $\delta(P)>1$, and moreover, the leftover environment satisfies $\delta(\text{leftover}(P)) = (\delta(P)-1)_+$ (where by $\text{leftover}(P)$ we mean the distribution of the cookie environment left over after the first walker has gone to infinity). By showing that the leftover environment satisfies the mentioned weaker version of (SE) whenever the original environment does, he concluded that if $k+1>\delta(P)>k$ then the procedure of walking on leftover environments can be repeated $k$ times, with all $k$ walkers $X^{(k)}$ being transient to the right a.s., while the $(k+1)$-st walker is recurrent a.s.\
The following theorem shows the same for the case of (IID), (BD) and (WEL):

\begin{theorem}[Exact criterion for transience]\label{thm:bddCaseTransience}
Assume that $P$ satisfies (IID), (BD) and (WEL). Then
$X^{(k)}_n\to + \infty$ a.s.\ if $k<\delta$, $X^{(k)}_n\to - \infty$ a.s.\ if $k<-\delta$, and is recurrent a.s.\ otherwise.
\end{theorem}

Note that while we assume the original measure $P$ over cookie environments is (IID), the leftover environments are no longer (IID), thus the last theorem is not within the scope of the results of \cite{basdevant2008speed} and \cite{kosygina2008positively} (unless $|\delta| < 1$).

We now move to the question of the speed of the walkers $X^{(k)}$. We first discuss a law of large numbers (LLN).

\begin{theorem}[Law of large numbers]\label{thm:ZernerLLN}
Assume that $P$ satisfies (SE). There are constants $v_k$, $k\ge 1$, such that $\frac{X^{(k)}_n}{n}\to v_k$ a.s.\
\end{theorem}

Since LLN holds, a natural question is if and when are the constants $v_k$ nonzero. We give an exact criterion for positive and negative speed in terms of $\delta$ for the leftover environments when the original environment satisfies (IID), (BD) and (WEL).

\begin{theorem}[Exact criterion for ballisiticity]\label{thm:bddCaseSpeed}
Assume that $P$ satisfies (IID), (WEL), and (BD), and define $v_k$ as in Theorem \ref{thm:ZernerLLN}. $|v_k|>0$ if and only if $k<|\delta|+1$, in which case $\delta$ and $v_k$ have the same sign.
\end{theorem}
To the best of our knowledge, the theorem gives the first examples of probability measures over cookie environments satisfying directional (SE) with positive speed that do not trivially follow by comparison to an (IID) measure with positive speed. To get (SE) examples we define in section \ref{sec:stationary} a stationary ergodic variation of leftover environments, to which all the above theorems hold.
Note that by Theorem \ref{thm:bddCaseTransience}, we already know that if $k\ge\delta$ then the walk $X^{(k)}$ is recurrent, and thus $v_k=0$.
Combining the two theorems on speed and transience gives that when starting from a probability measures over cookie environments satisfying (IID), (BD), and (WEL) with expected drift per site $k<\delta\leq k+1$ ($k<-\delta\leq k+1$), the first $(k-1)$-st walkers will have strictly positive (negative) speed, the $k$-th walker will be transient to the right (left) with $0$ speed, and all subsequent walkers will be recurrent (all statements holding with probability $1$).

In order to analyze the walk on the $k$-leftover environment, we introduce the notion of an excited ``mob" - a set of $k$ walkers moving on the environment according to some given scheduling between them. We introduce a natural coupling between different processes on the same environment and use it to show that many properties of the movement of the mob are invariant under a wide choice of the scheduling. We then use the freedom we have in moving the mob to translate many of the techniques used for excited random walks to mobs, and finally to prove the main results on the leftover environments.

\subsection{Structure of the paper}
The paper is organized as follows:
In section \ref{sec:arrows} we discuss excited random walks in arrow environments, introduce mob walks and develop their basic theory which will be used throughout the paper. We then analyse properties of the $k$-minimum walk, a canonical way to move a mob of $k$ walkers.
In section \ref{se:ppcokies} we prove a 0-1 law for directional transience for the leftover environments (Theorem \ref{thm:TransienceThreshold}).
In section \ref{sec:LLN} we prove the Law of Large Numbers for leftover environments (Theorem \ref{thm:DLLN}).
In section \ref{sec:BPwM} we discuss some known results concerning branching processes with migration.
In section \ref{sec:transienceLeftover} we move to dealing with the (IID) and (BD) case and show the exact criterion for transience vs. recurrence for mob walks and for walks on leftover environments, Theorem \ref{thm:bddCaseTransience}.
In section \ref{sec:SpeedForkmin} we prove an exact criterion for positive speed of the minimum walk in the (IID) (BD) case.
In section \ref{sec:SpeedLeftover} we give a formula relating the speed of the minimum walks and the walkers on the leftover environments, and prove Theorem \ref{thm:bddCaseSpeed}.
In section \ref{sec:stationary} we define  a stationary ergodic version of the leftover environments, and show that the main theorems of this paper hold also for these environments.
Finally, in section \ref{sec:Remarks} we state some further remarks and open problems.
\section{Combinatorial Perspective - arrows}\label{sec:arrows}
\subsection{Arrow environments, Local time and Leftovers}

Given a cookie environment $\cookieenv$, we can realize $\cookieenv$ into a (random) list of arrows, or instructions, which tell the walker in which direction to walk in every possible visit to any position. This is done by a priori flipping an independent $\cookieenv(x,n)$-coin for each $x\in \Z$ and $n\in \N$. If we then let a walker walk according to this list of instructions, the law of a (non-random) walk walking according to the (random) list of instructions sampled from $\omega$ is the same as the quenched law of the excited random walk on $\omega$. This leads us to the definition of arrow environments given below. Arrow environments were studied in the work of Holmes and Salisbury \cite{holmes2012combinatorial}, where they were called ``arrow systems". They considered partial orderings of these arrow systems and used them to couple cookie environments. They then deduced various monotonicity results on excited random walks and related models. Using related coupling techniques Peterson \cite{peterson2012strict} managed to strengthen previous monotonicity results of Holmes and Salisbury \cite{holmes2012combinatorial} and of Kosygina and Zerner \cite{kosygina2008positively} and proved strict monotonicity of speed and return probability to $0$ of excited random walk with respect to the cookie strengths.
We take a different route and use arrow environments mainly as a natural way to couple various processes on the same cookie environment, which will prove key to our analysis of the leftover environments. Considering different processes on the same arrow environment allows us to distill the ``combinatorial" part from some of the probabilistic arguments regarding ERW.

In this section, we restrict ourselves to combinatorial aspects of arrow environments, while the connection to cookie environments and probabilistic arguments is deferred to the following sections.

\begin{defi}\label{d:arrow}
An \emph{arrow environment} is an element $a\in\{-1,1\}^{\Z \times \N}$.
A \emph{walk} on an arrow environment $a$ and its \emph{local time} are two sequences $X_t\in \Z$, $t\geq 0$, and $L_t: \Z \to \N$, $t\geq 0$, defined by
$$X_0=x,\, L_0\equiv 0,\, \ L_{t+1}(x) = \#\{1\le s\le t : \ X_s=x\},x\in\Z,\,     X_{t+1}=X_t+ a ( X_t , L_{t+1}(X_t)).$$
\end{defi}

We will assume $x=0$ unless specified otherwise. This associates to each arrow environment a well-defined non-random walk and its local time.

\begin{defi}
Given an arrow environment $a$, the \emph{asymptotic local time} $L:\Z \rightarrow \N\cup\{\infty\}$ is given by
\begin{equation}
L(x) = \lim_{t\rightarrow\infty}L_t(x).
\end{equation}
Note that $L$ is well-defined since $L_t(x)$ is non-decreasing in $t$ for each $x\in\Z$.
\end{defi}

We will require the following non-degeneracy condition to insure the walk $X$ does not stay within a finite interval:
\begin{defi}
An arrow environment $a\in\{-1,1\}^{\Z \times \N}$ is called \emph{non-degenerate} if for each $x\in \Z$ $a(x,n)\neq a(x,n+1)$ for infinitely many $n$.
The set of all non-degenerate arrow environments is denoted by $A$.
\end{defi}

\begin{remark}\label{rem:IfForOnLIsInfThenForAllLIsInf}
For $a\in A$, if $L(x)=\infty$ for some $x\in\Z$, then the non-degeneracy condition implies $L(x\pm1)=\infty$. Iterating the last argument gives that
the asymptotic local time $L$ of the walk on $a$ is either finite for all $x\in \Z$ or infinite for all $x\in \Z$.
\end{remark}

\begin{defi}
An arrow environment $a\in A$ is called \emph{transient} if the asymptotic local time of the walk on $a$ is finite everywhere.
Define $A_1\subset A$ to be the set of transient arrow environments.
\end{defi}
By Remark \ref{rem:IfForOnLIsInfThenForAllLIsInf}, $a\in A$ is transient if and only if the walk's asymptotic local time is finite at \emph{some} location.
Also, $a\in A$ is transient if and only if $\lim_{t\to \infty} X_t = +\infty$ or $\lim_{t\to \infty} X_t = -\infty$.

Given $a\in A_1$, consider the (first) \emph{leftover environment} $LO = LO(a)$ defined by
\begin{equation}\label{eq:LeftoverEnv}
LO(x,n) := a(x,n+L(x))
\end{equation}
(i.e.\ we dispose of all ``arrows" that have been used by the walk).
Consider now the leftover environment. If it is also transient, then we may define the environment that is left over from a second walk on that environment.
In fact, we can iterate this as long as the leftover environments we get are all transient. More formally, let $A_0=A$ and define by induction for $j\geq 1$
$A_j\subset A_{j-1}$ to be the set of all $a\in A_{j-1}$ such that $LO(a)$ is transient.
An arrow environment $a\in A_k$ is called \emph{$k$-transient}. 
If $a$ is $k$-transient then we define recursively
\begin{equation}\label{eq:jthLeftover}
a_0=a,\ a_j=LO(a_{j-1}),\ k\ge j \ge 0.
\end{equation}
A way to think about the leftover environments is to imagine $k$ ``walkers" starting at the origin on an arrow environment $a\in A_{k-1}$.
We first take one of the walkers, and walk it according to the arrow environment $a$ an infinite number of steps until it ``goes to infinity".
We then take the next walker and move it according to the environment $a_1$ left after the first walker passed through.
After the second walker ``went to infinity" also, we are left with an environment $a_2$ on which the third walker moves and so on.
Thus we may think of this as ``sequentially" walking the $k$ walkers on $a$. This is well-defined as long as $a\in A_{k-1}$, and the last walker (the $k$-th) will also have everywhere-finite asymptotic local time
if and only if $a\in A_k$. In view of the above discussion we shall define $L^{(k-\mathbf{seq})}$ to be the total local time achieved by the $k$ walkers on the environment $a$. That is, if $L^{(j)}$ is the local time of the walk on the environment $a_{j-1}$ (see \eqref{eq:jthLeftover}) $j=1,...,k$, then
\begin{equation}\label{eq:Lkseq}
L^{(k-\mathbf{seq})}=\sum_{j=1}^{k}L^{(j)}.
\end{equation}

In the above description the walkers move ``sequentially" one after the other ``finished" moving. This motivates a study of other ways to move $k$ walkers on a given arrow environment.
In the remainder of the section we will study how the walkers may be moved together in some arbitrary order, and provide conditions under which the asymptotic local time is invariant to this choice of ordering.

\begin{remark}
Throughout this paper we assume that all of the walkers are initially located at the origin, with the exception of Section \ref{sec:stationary}. This assumption is made for a convenient presentation. However, we wish to stress that all definitions and propositions can be easily transformed so that the walks will have different (yet fixed in advance) initial positions, with no additional arguments.
\end{remark}

\subsection{Mob walks and exchangeability}

In this section we consider a ``mob" of $k$ particles on the arrow environment - that is moving them one at a time by choosing which one should move at each step.
(Note that we call use the terms `walker' and `particle' indistinguishably.)

Fix an arrow environment $a\in A$ and a function $S:\N_0\to \{1,...,k\}$. We call the function a $S$ \emph{$k$-scheduling} and define the $S$-mob walk $X=(X^{(1)},...,X^{(k)})$ and the local time $L^{(S)}$ by:
$$X^{(i)}_0=0 \text{ for all } 1\le i\le k,\ L_0\equiv 0$$ and for $t\ge 0$

$$ L_{t+1}^{(S)}(x)=\# \{0\le s \le t : \text{ there is some } 1\le i\le k \text{ such that } X^{(i)}_s = x \text{ and } X^{(i)}_{s+1} \neq x \}$$\label{def:LocalTime}
and
$$X^{(j)}_{t+1} =
\begin{cases} X^{(j)}_{t} + a(X^{(j)}_{t},L_{t+1}(X^{(j)}_{t})) & \text{ if } j=S(t)\\
X^{(j)}_{t} & \text{ if } j\ne S(t).
\end{cases}
$$
Note that $ L_{t+1}^{(S)}(x)=\# \{0\le s\le t : X^{(S(s))}_s = x  \}$.

Define the asymptotic local time $L^{(S)}$ by $L^{(S)}(x)=\lim_{t\to\infty}L_t^{(S)}(x)$.

\begin{remark}\label{rem:koneforallallforone}
Given any $S:\N_0\to \{1,...,k\}$, one can generalize Remark \ref{rem:IfForOnLIsInfThenForAllLIsInf} to $S$-mob walks. The same arguments give that for $a\in A$ the $S$ - asymptotic local time
is either finite everywhere or infinite everywhere.
\end{remark}

To avoid degeneracies, we require that each particle is chosen infinitely often:
\begin{defi}
A function $S:\N_0\to \{1,...,k\}$ is called a \emph{proper $k$-scheduling} if $|S^{-1}(i)|=\infty$ for all $1\le i\le k$.
\end{defi}

It is useful to note that one may choose a $k$-scheduling in a way that will depend on a given arrow environment. A useful class of scheduling are the ``algorithmic" scheduling, in which the particle to move next is chosen according to a deterministic function (``algorithm") of the history of the $k$-mob walk until that time.
Since given an arrow environment, there is no randomness involved, any such ``algorithm" defines a function $S:\N_0\to \{1,...,k\}$, and for it to be a proper $k$-scheduling one just has to make sure
that the condition $|S^{-1}(i)|=\infty$ holds for all $1\le i\le k$. Note that in such scheduling for each $n$, $S(n)$ depends only on a finite number of arrows in the environment, which ensures that given a measure over arrow environments, the $k$-mob walk is a stochastic process. (See Section \ref{se:ppcokies}.)

The next lemma shows that the asymptotic local time is invariant under the choice of proper $k$-scheduling.

\begin{lemma}\label{lem:exchangeabilty}
Let $a\in A$ and let $S$ be a proper $k$-scheduling. Then for any other proper $k$-scheduling $S'$, $L^{(S)}(x)= L^{(S')}(x)$ for all $x\in\Z$.
\end{lemma}

The proof of Lemma \ref{lem:exchangeabilty} is an adaptation of Proposition 4.1 of \cite{diaconis1991growth}. The main difference is that in our case we do not assume a finite termination, but instead we require the asymptotic local time to be everywhere finite.
Before we get to the proof, we introduce some notion.
Define $I_j^{(S)}(x)$ to be the in-degree of the vertex $x$ at time $t$ by the $S$-mob walk $X$. That is,
$$I_0^{(S)}(x)= \begin{cases} k  & \text{ if } x=0 \\ 0 & \text{otherwise } \end{cases}\ \, \text{ and } \, \
I_{t+1}^{(S)}(x)=\{0\le s\le t: X_{s+1}^{(S(s))}=x\} + I_0^{(S)}(x).$$
Let $I^{(S)}(x)=\lim_{t\to\infty} I_t^{(S)}(x)\in\N_0\cup\{\infty\}$.
We note that $I_t^{(S)}(x)\ge L_t^{(S)}(x)\ge I_t^{(S)}(x)-k$ for all $t\ge 0$, and if $S$ is a proper scheduling then every visit in $x$ is eventually followed by getting out from $x$, namely $I^{(S)}(x)=L^{(S)}(x)$.
Note also that for each $t$ and $x$, if $I_t^{(S)}(x) = L_t^{(S)}(x)$, it means that there are no particles in place $x$ at time $t$, that is $X^{(i)}_t\ne x$ for all $1\le i \le k$.

\begin{proof}[Proof of Lemma \ref{lem:exchangeabilty}]
We shall first prove the lemma in case that the local time $L^{(S)}(x)$ is finite for all $x\in\Z$.
It is enough to show an inequality: $L^{(S)}(x)\ge L^{(S')}(x)$ for all $x\in\Z$. Indeed, by interchanging the roles of $S'$ and $S$ and reusing the same statement, we get the other inequality.
Assume towards a contradiction that there is some $x\in\Z$ such that $L^{(S)}(x)< L^{(S')}(x)$. Since $L_0^{(S)}\equiv L_0^{(S')}\equiv 0$ and both functions are non-decreasing with respect to time in each position, we may fix $j$ to be the minimal time index so that there is some $x\in\Z$ such that $L^{(S)}(x)< L_{j+1}^{(S')}(x)$.
Denote by $v$ the position of the particle that was moved under the $S'$-scheduling at time $j$, that is $v=X'^{(S'(j))}_j$, where $X'$ is the $S'$-mob walk. Then $L^{(S)}(v)< L_{j+1}^{(S')}(v)$ and $L^{(S)}(v)= L_{j}^{(S')}(v)$, as the local time at time $j+1$ differs from the local time at time $j$ by adding $1$ at exactly one position.
By minimality of $j$, $L^{(S')}_{j}(S'(s)) \le L^{(S)}(S'(s))$ for all $0\le s\le j$. Therefore the in-degree of $v$ at time $j$ by the walk $X'$ is bounded from above by the in-degree of $v$ at time $\infty$ by the walk $X$. Therefore $I_j^{(S')}(v)\le I^{(S)}(v)=L^{(S)}(v)=L_j^{(S')}(v)$, and so $I_j^{(S')}(v) = L_j^{(S')}(v)$. Hence $X'^{(S'(j))}_j\ne v$, a contradiction. Therefore the lemma is proved when $L^{(S)}$ is everywhere finite.

Assume now that $L^{(S)}(x)=\infty$ for some $x$. By Remark \ref{rem:IfForOnLIsInfThenForAllLIsInf} it holds in this case that $L^{(S)}\equiv\infty$. Now if by contradiction $L^{(S')}(x')<L^{(S)}(x')=\infty$ for some $x'$, then using again Remark \ref{rem:IfForOnLIsInfThenForAllLIsInf}, we conclude that $L^{(S')}(x'')<\infty$ for all $x''\in\Z$. But then applying the lemma for the finite local time case while interchanging the roles of $S'$ and $S$ would imply that $L^{(S)}(x')=L^{(S')}(x')<\infty$, a contradiction.
\end{proof}

\begin{remark}\label{rem:exchangeability}
\begin{enumerate}
    \item The same proof gives the following variation of the last lemma. For any set $I\subset \Z$, and any proper $k$-scheduling $S$, if we stop each particle once it exits $I$ (i.e.\ those particles do not move even if ``chosen" by the scheduling) then the asymptotic local time in $I$ is independent of the choice of $S$ as long the local time is finite (everywhere) in the set $I$ for some proper $k$-scheduling. For finite sets this is a special case of Diaconis-Fulton \cite[Proposition 4.1]{diaconis1991growth}).

    \item The proof of the last lemma also applies to any bounded degree graph with the appropriate definitions of arrow environments, local time and non degeneracy conditions.

    \item \label{rem:nonproper} If we remove the condition that the $k$-scheduling $S'$ is proper, we get an inequality $L^{(S')}(y)\leq L^{(S)}(y)$ for all $y$. Indeed, otherwise there is some $y$ and a minimal $t$ so that $L^{(S')}_t(y)>L^{(S)}(y)$ (and in particular, as $a\in A$, $L^{(S)}$ is finite everywhere). We can define now a new proper $k$-scheduling $S''$ by using $S'$ until time $t$ and then going over the $k$ particles in periodic order. We get that $L^{(S'')}(y)\geq L^{(S')}_t(y)>L^{(S)}(y)$, contradicting the Lemma.
\end{enumerate}
\end{remark}
We conclude the section by showing that the leftover environment of ``sequentially" walking the $k$ walkers on $a$ cannot be changed by choosing any other proper scheduling. Remember that $L^{(k-\mathbf{seq})}$ was defined in \eqref{eq:Lkseq}.

\begin{theorem}[Exchangeability]\label{thm:exchangeability}
For any $k$-transient environment $a\in A_k$ and any proper $k$-scheduling $S$, the local time $L^{(S)}(x)=L^{(k-\mathbf{seq})}(x)$ for all $x\in\Z$.
\end{theorem}

\begin{proof}
As $L^{(k-\mathbf{seq})}$ is everywhere finite by the assumption, then by Lemma \ref{lem:exchangeabilty} it is enough to show that there is at least one proper $k$-scheduling $S$ for which $L^{(S)}=L^{(k-\mathbf{seq})}$. This is the content of the next lemma.
\end{proof}

\begin{lemma}\label{le:sequencial}
        Let $a\in A_{k}$ be a $k$-transient arrow environment. Then there is a proper $k$-scheduling $S$ satisfying the equality $L^{(S)}(x)= L^{(k-\mathbf{seq})}(x)$ for all $x\in\Z$. %In particular, $L^{(S)}(x)$ is finite for all $x\in\mathbb{Z}$.
\end{lemma}

\begin{proof}
First, let us describe the idea of the proof informally. We want each walker to walk ``as if" the walkers before it were sent to infinity, thus we want the arrow environment the $j$-th walker sees at
each step to be the same as that in $a_{j-1}$. To this end, before each move of the $k$-th walker, we first move the first walker a very large number of steps so the leftovers look like $a_1$.
We then move the second walker some large number of steps to turn into $a_2$ and so on while making sure it walks always still on an environment that looks like $a_1$.
We repeat this for $k-1$ walkers until we are sure the $k$-th particle sees at its position the same environment as in $a_{k-1}$. We then make a single move with the $k$-th particle and repeat the whole process.
We now give an explicit construction:
    Since $a\in A_{k-1}$, we can define $g_i(x)$ for each $i=1,...,k-1$, to be the minimal $s$ so that for all $t\ge s$ the walk $X^{(i)}_t$ in the arrow environment $a_{i-1}$ satisfies
    $|X^{(i)}_t|> x$, and let $g_k(y):= y$ be the identity function.
    Set $f_i(\cdot) := g_{i}\circ g_{i+1}\circ\ldots\circ g_{k}(\cdot)$ for $i=1,\ldots,k$.
    We shall use the $f_i$ to define the requested proper $k$-scheduling $S$. We build the sequence $S$ in consecutive blocks $(B_n)_{n\ge 1}$ in the following way:
    The first block $B_1$ will consist of $f_1(1)$ $1$'s followed by $f_2(1)$ $2$'s and so on until it ends with $f_k(1)$ $k$'s.
    The blocks $B_n$ for $n\geq 2$ are given by a sequence of $f_1(n)-f_1(n-1)$ $1$'s followed by $f_2(n) - f_2(n-1)$ $2$'s and so on until ending with $f_k(n)-f_k(n-1)$ $k$'s.
    Note that in our construction $f_k(n)=n$ so each block ends with a single instance of $k$.
    It is clear from the definition that $S$ is indeed a proper $k$-scheduling. A simple induction now shows that the $j$-th particle always moves in places that particles $1$ to $j-1$ will never visit again, and therefore it sees the same arrows as in $a_{j-1}$.
\end{proof}

Theorem \ref{thm:exchangeability} and Remark \ref{rem:koneforallallforone} give us the following characterizations of $k$-transience.
\begin{remark}
Fix an arrow environment $a\in A$.
The following are equivalent:
\begin{enumerate}
\item $a$ is $k$-transient.
\item There is some proper $k$-scheduling $S$ and some $x\in \Z$ such that the asymptotic local time of the $S$-mob walk is finite at $x$.
\item For all proper $k$-scheduling $S$ and all $x\in \Z$ the asymptotic local time of the $S$-mob walk is finite at $x$.
\end{enumerate}
 \end{remark}
\subsection{The minimum walk}

We saw above that given $a\in A_k$, if one wishes to consider the environment $a_k$ then it is enough to consider the asymptotic local time of \emph{some} proper $k$-scheduling. We next present a canonical way to produce such a $k$-scheduling.

\begin{defi}
Given an arrow environment $a\in A_{k-1}$, define \emph{the minimal $k$-scheduling} $S^{(\min)}:\N\to\{1,\dots,k\}$ defined inductively together with the corresponding $S^{(\min)}$-mob walk $X$ by: $(X^{(1)}_0,...,X^{(k)}_0)=(0,...,0)$ and
\begin{equation}\label{eq:minimumScheduling}
S^{(\min)}(t)=\min\arg\min\{X^{(1)}_t,...,X^{(k)}_t\}.
\end{equation}
(I.e., the walk is defined by moving at each step the leftmost particle, breaking ties using the order of the particles.)
\end{defi}

We will call the $S^{(\min)}$-mob walk \emph{the minimum walk} on $k$ particles, or the $k$-minimum mob walk, and denote it sometimes by $X^{(\min)}$. We denote by $X_t=\min_{1\leq j\leq k}X^{(j)}_t$ the position of the leftmost particle under the minimum scheduling at time $t$, and will often think of $X=(X_t)_{t\ge0}$ as a nearest neighbor walk on $\Z$ which may also stay in its place. $X$ will be therefore called the $k$-minimum walk.

\begin{remark}\label{rem:Minproperties}
If $a\in A$ then by Remark \ref{rem:koneforallallforone} we get that the $k$-minimum walk satisfies \linebreak $\liminf _{n\to\infty}X_n,\limsup _{n\to\infty}X_n\in\{-\infty,+\infty\}$.
Observe that by the definition of the $k$-minimum walk, if $X_n=r$ for some $n$, then at most one particle can be to the left of $r$ at any time $m>n$.
In fact, at any given time, $k-1$ of the particles are always partitioned between the rightmost point of the the $k$-minimum mob walk reached so far and the point to its right, and every time a particle makes a jump to the left, it will continue to be the ``active" particle until it either returns to the pack, or drifts away forever. Since the non-degeneracy condition on the arrow environment $a\in A$ ensures that a particle cannot remain caught in a finite interval, if the particle does not return to the pack after going leftward, it must drift to $-\infty$.
It follows that $S^{(\min)}$ is a proper $k$-scheduling if and only if $X_n \nrightarrow -\infty$.
\end{remark}

\subsection{Transience versus Recurrence}\label{subsec:deterministicTransienceVsRecurrence}

\begin{def}\label{def:kRightTrans}
An arrow environment $a\in A$ is called \emph{$k$-right transient} if $a_{i-1}$ are transient to the right, $1\le i\le k$ (i.e., the walks $(X^{(i)}_t)_{t\ge 0}$ on the environment $a_{i-1}$ satisfy $\lim _{t\to \infty} X^{(i)}_t=+\infty$, $1\le i\le k$).
\end{def}

The Exchangeability Theorem \ref{thm:exchangeability}, together with Remark \ref{rem:Minproperties} on the $k$-minimum walk, gives the following immediate corollary:
\begin{cor}\label{cor:seqRightTranEquivMinRightTan}
An arrow environment $a\in A$ is $k$-right transient if and only if the $k$-minimum walk $X=(X_t)_{t\ge0}$ on $a$ is transient to the right.
\end{cor}

\begin{proof}
If the $k$-minimum walk is transient to the right, then in particular $S^{(\min)}$ is a proper $k$-scheduling and hence by the Exchangeability Theorem \ref{thm:exchangeability} the local time of the $k$-sequential walk is finite everywhere and identically zero for all negative small enough $x$. Thus $a$ is $k$-right transient. For the other implication, Remark \ref{rem:exchangeability} \eqref{rem:nonproper} together with Theorem \ref{thm:exchangeability} tells us that the local time of the $k$-minimum walk is bounded from above by the local time of the $k$-sequential walk, which is everywhere finite and equals zero for all negative small enough $y$. In particular, the $k$-minimum walk is transient to the right.
\end{proof}

Modifying Kosygina and Zerner's argument in Section 3 of \cite{kosygina2008positively} we get a condition for $k$-right transience of $a$, whom we shall discuss now.
Associate to each arrow environment $a\in A$ two deterministic processes. The first one, $X=(X_t)_{t\ge 0}$ is the $k$-minimum walk defined in the paragraph before Remark \ref{rem:Minproperties}. The other one, $z=(z_n)_{n\ge 0}$, is given by:

\begin{equation}\label{eq:defOfz}
z_0=k;\text{ and } z_{n+1} = \max\left\{0,\inf \{ t : \sum_{i=1}^t (1-a(n,i))= z_n-(k-1) \} - (z_n-(k-1))\right\}.
\end{equation}
(Here, as a convention, $\sum_{i=1}^{-t}(\cdots)\equiv 0$ for $t\ge 0$.)

Note that $z_{n+1}$ is defined to be the number of $1$'s in $a(n,\cdot)=(a(n,1),a(n,2),\dots)$ before there are $(z_n-(k-1))$ $0$'s in this sequence if $z_n\geq k$, and zero otherwise. As we will show below, the process $z$ includes all the information regarding the return of at least one particle of the $k$-min mob walk to the origin (more accurately, the information regarding the hitting time at $-1$). Let
\begin{equation}\label{eq:defOfTminusOne}
t_{-1}:=\inf\{t\ge 0:X_t=-1\}
\end{equation}
to be the hitting time of the $k$-minimum walk at $-1$.
Define
\[ w_0=k; \text{ and } w_n=\# \{t<t_{-1}: X_t=n-1 \text{ but } X_{t+1}\neq n-2\},\text{   } n\ge 0.\]

We first observe some simple but useful properties of the $k$-minimum walk $x$.
Define $M:=\sup \{ X_t: t < t_{-1}\}\in\N\cup\{\infty\}$.

\begin{lemma}\label{lem:ComparingCrossings} Fix $a\in A$. The following hold for all $0\le n\le M$.
\begin{enumerate}
\item \label{c:1}If $t_{-1}=\infty$ then $w_n=k\ + \text{ total left crossings of }(n,n-1)\text{ by the $k$-min mob walk}.$ \label{eq:ComparingCrossingsTranCase}
\item \label{c:2} If $\ t_{-1}<\infty$ then $w_n = k-1\ + \text{ the number of left crossings of } (n,n-1) \text{ by the $k$-min mob walk}$ up to and including time $t_{-1}$. Moreover, in this case $w_{M+1}\le k-1$. \label{eq:ComparingCrossingsRecCase}
\end{enumerate}
\end{lemma}

\begin{proof}
 We first show \eqref{c:1}. For $n=0$ this is trivial since the edge $(0,-1)$ was never crossed and $w_0=k$ by definition. Therefore we may assume $n>0$.
 Since $a\in A$ and $t_{-1}=\infty$, Remark \ref{rem:koneforallallforone} implies that all local times are finite. Moreover, as the local time at $-1$ is zero, then $X_t\to\infty$. In particular there is some time $N$ such that $X_t > n$ for all $t>N$, so the edge $(n-1,n)$ is never crossed (in either direction) after time $N$, and thus $w_n$ is finite.
 For any $t>N$ we have that all $k$ particles are to the right of the edge $(n-1,n)$. Since all $k$ started to the left of $(n-1,n)$ it follows that each particle had one more right-crossing of that edge then left-crossings. Summing over all particles finishes the proof.
To see \eqref{c:2} we note that at time $t_{-1}$ all particles but one are either at $M$ or $M+1$ (only one particle may move left of $M$ at any time - see remark \ref{rem:Minproperties}). Thus at time $t_{-1}$, $k-1$ particles are to the right of the edge $(n-1,n)$ and one (the one that hit $-1$) is to the left of it.
The proof follows as in clause \eqref{c:1}.
\end{proof}

\begin{lemma}\label{lem:comparingZandW}
For all $n\ge 0$, the following hold.
\begin{enumerate}
\item If $t_{-1}<\infty$ then $z_n =w_n$. \label{ZEqualsRightCrossings}
\item If $t_{-1}=\infty$ then $z_n \geq w_n$. \label{ZisGraterThanRightCrossings}
\end{enumerate}
\end{lemma}

\begin{proof}
Assume first that $t_{-1}<\infty$. We will prove by induction on $n<M$ that $z_n = w_n$. For $n=0$ we have $z_0=k=w_0$. Assume now that $z_n=w_n$. Since $t_{-1}<\infty$, then the last crossing of $x$ before time $t_{-1}$ of the undirected edge $(n,n+1)$ is a left crossing, and therefore $a(n,L_{t_{-1}}(n))=0$. This implies that
the number of $0$'s in $\{a(n,1),...,a(n,L_{t_{-1}}(n))\}$ equals the total number of left crossings of $(n,n-1)$ before time $t_{-1}$. Note that by Lemma \ref{lem:ComparingCrossings} the last quantity equals $w_n-(k-1)$. As $z_n=w_n$, then $z_n-(k-1)$ is the number of zeros in $\{a(n,1),...,a(n,L_{t_{-1}}(n))\}$.
Now, $w_{n+1}=\# \{t<t_{-1}: X_t=n \text{ but } X_{t+1}\neq n-1\}$, which is the number of $1$'s in $a(n,\cdot)$ before the last visit of $x$ there, that is the number of ones in $\{a(n,1),...,a(n,L_{t_{-1}}(n))\}$. The latter is exactly the number of $1$'s before $z_n-(k-1)$ $0$'s in $a(n,\cdot)$, which is $z_{n+1}$ by definition.

Consider now the case $t_{-1}=\infty$. Again, we will prove by induction on $n<\infty$ that $z_n \geq w_n$. For $n=0$ we have $z_0=k=w_0$. Assume by induction that $z_n\ge w_n$. As $a\in A$ the process $x$ is transient, so every vertex (and edge) has a finite last time when it was visited by the walk $x$. By Lemma \ref{lem:ComparingCrossings} $w_n$ equals $k-1$ plus the number of $0$'s in $\{a(n,1),...,a(n,L_\infty(n))\}$. Using the definition of $z_{n+1}$ and the induction hypothesis it follows that $z_{n+1}$ is greater than or equal to the number of $1$'s in $\{a(n,1),...,a(n,L_\infty(n))\}$, which is $w_{n+1}$ by definition.
\end{proof}

As a result, we get the next theorem.
\begin{theorem}\label{thm:tzn}
$t_{-1}<\infty$ if and only if $z_n=0$ for some $n$.
\end{theorem}
\begin{proof}
If $t_{-1}<\infty$, then by Lemma \ref{lem:comparingZandW} together with the ``moreover" part of Lemma \ref{lem:ComparingCrossings}.\ref{eq:ComparingCrossingsRecCase} we get that $z_{M+1}=w_{M+1}\le k-1$ and hence $z_{M+2}=0$.

On the other hand, if $z_n=0$ then by Lemma \ref{lem:comparingZandW}.\ref{ZisGraterThanRightCrossings} we have that $0\le w_n\le z_n =0$. This shows that $t_{-1}<\infty$ since otherwise it would contradict the fact that by Lemma \ref{lem:ComparingCrossings}.\ref{eq:ComparingCrossingsTranCase}, $w_n\ge k$ holds.
\end{proof}

\section{Probabilistic Perspective - Cookies}\label{se:ppcokies}
Every cookie environment $\om$ may be considered naturally as a product measure over the space of arrow environments. Recall that given an arrow environment, the behavior of the processes we study is deterministic, thus the $k$-minimum walk $X_n$ (with given initial condition), the process $z_n$ and the stopping times $t_{\pm 1}$ are deterministic functions of the arrow environment.
It will therefore be convenient to regard the quenched measure $\QuenchedCookies$ as a measure on arrow environments (from which all the above quantities are derived) and to consider the annealed measure on arrow environments $\mathbb{P}$. It will sometime be convenient to allow the walkers not to start at $0$ but rather at some other point on $\Z$. We denote by $\mathbb{P}_m$ the same annealed measure on arrow environments with all the walkers initially positioned at $m$.

The annealed measure $\mathbb{P}$ over arrow environments inherits many of the properties of the measure $P$ on cookie environments.
It is straightforward that if $P$ is i.i.d.\ then so is $\mathbb{P}$. This also holds for (SE):
\begin{lemma}\label{le:arrows ergodic}
If the measure $P$ on cookie environments is stationary and ergodic, then so is $\mathbb{P}$
\end{lemma}
\begin{proof}
Stationarity is straightforward.
To get ergodicity, note that one can derive the arrow environment by attaching an independent uniform $[0,1]$ random variable to each cookie, and comparing it to the bias of the cookie.
Therefore the arrow environment is a factor of the product of the environment $e$ and an i.i.d.\ collection of uniform $[0,1]$ random variables. Since an i.i.d.\ collection is mixing, the product is ergodic and therefore the annealed measure on the space of arrow environments, which is a factor of the product is also ergodic.
\end{proof}

Other properties of the measure $P$ also translate directly to properties of $\mathbb{P}$. In particular, by Borel-Cantelli's lemma $\QuenchedCookies(a\in A)=1$ if and only if  $\QuenchedCookies$ satisfies (ND).

\subsection*{Zero-one Laws and the proof of Theorem \ref{thm:TransienceThreshold}}
The remainder of this section is dedicated to proving 0-1 laws for directional transience and recurrence of the walker on $j$-left over cookie environment. As a corollary we will derive Theorem \ref{thm:TransienceThreshold}. Notice that we assume that the probability over the cookie environments satisfies (SE), (WEL) and (ND). For clarity, we shall mention in each statement below which assumptions on the environments are required.

Throughout this section we fix $k$ to be the number of particles, $X$ is assumed to be the (now random) $k$-minimum walk, and for $m\in\Z$ we let $T_m$ be the random hitting time of $m$ by the walk $X$. Recall that the (random) process $Z_n$ is defined to have initial value $Z_0=k$.

\begin{lemma}\label{lem:T-1iffZ} For every probability measure over cookie environments satisfying (SE) and (ND) the following holds: For almost every $\om\in\Om$, $\QuenchedWalks(T_{-1}=\infty ) > 0$ if and only if $\QuenchedCookies(Z_n>0 \text{ for all } n) > 0$
\end{lemma}

\begin{proof}
This follows directly from the deterministic case, Theorem \ref{thm:tzn}.
\end{proof}

A \emph{right-excursion} of a walk $x$ is a sequence of moves $X_{\tau_0},\ldots, X_{\tau_1}\le \infty$ such that $X_{\tau_0}=0$, either $X_{\tau_1}=0$ or $\tau_1=\infty$, and $X_t>0$ for all $\tau_0<t<\tau_1$.
 Call $m\geq 0$ an \textbf{optional regeneration position} for an arrow environment $a$ if the $k$-minimum walk $X_n$, with all particles started at $m$ never hits $m-1$, that is if $t_{m-1}=\infty$.
 We call $m\geq 0$ a \textbf{regeneration position} if in addition, when starting the particles from $0$, the $k$-minimum walk $x$ reaches $m$ after some finite time.
 Note that in the $k$-minimum walk, no particle will move from position $m$ until all $k$ particles reach position $m$ (that is until the $k$-minimum walk $X_n$ reaches $m$), which means that the arrow environment on $[m,\infty)$ remains unchanged until all particles reach $m$ (if they ever do). It follows that if $m$ is an (optional) regeneration position, and the $k$-minimum walk $X_n$ reaches $m$, then it will afterwards never return to $m-1$.

 \begin{lemma}\label{le:inf_opt_reg} Let $P$ be a probability measure over cookie environments satisfying (SE).
 If $\mathbb{P}_0(T_{-1}=\infty)>0$ then there are a.s.\ infinitely many optional regeneration positions.
 \end{lemma}
 \begin{proof}
     Let $p:=\mathbb{P}_0(T_{-1}=\infty)>0$.
      By stationarity of the arrow environment, $\mathbb{P}_m(T_{m-1}= \infty) = p$ for any $m\geq 0$. By the ergodic theorem, we have that
    $$
    \frac{1}{n}\sum_{m=1}^n \mathbf{1}_{\{m \text{ is a optional regeneration position}\}} \to p \text{ a.s.}.
    $$
    In particular there are a.s.\ infinitely many optional regeneration positions.
 \end{proof}

 The following lemma is similar to a part of Lemma 8 of \cite{kosygina2008positively}, which was proved for $k=1$ in the (IID)\ case. See also \cite{amir2013zero} section 2.3.

\begin{lemma}\label{lem:TisInfiniteFinitelyManyExcursions} Let $P$ be a probability measure over cookie environments satisfying (SE) and (ND).
If $\mathbb{P}_0(T_{-1}=\infty)>0$ then there are a.s.\ only finitely many right excursions.
\end{lemma}
\begin{proof}
On $\limsup X_n<\infty$, since $a\in A$ a.s., Remark \ref{rem:Minproperties} yields that $X_n\to -\infty$ and in particular the number of right excursions is a.s.\ finite. On $\limsup X_n=\infty$, since by Lemma \ref{le:inf_opt_reg} there a.s.\ exist (infinitely many) optional regeneration positions, the walker a.s.\ hits such a position $m$ and from that time on it will never return to $m-1$ let alone $0$, and thus there are only finitely many right excursions.
\end{proof}

In particular, we get the following corollary:

\begin{cor}\label{cor:CommingBack-k-walkIsTransient} Let $P$ be a probability measure over cookie environments satisfying (SE) and (ND).
If $\mathbb{P}_0(T_{-1}=\infty)>0$ then $\mathbb{P}_0(X_n=0\text{ i.o.})=0$.
\end{cor}

Also the following lemma is almost identical to another part of Lemma 8 of \cite{kosygina2008positively}. The proof we shall present here is a variant a finite modification argument used in \cite{kosygina2012excited} under the assumption (ELL), which extends also to the (IID) and (WEL) case.

\begin{lemma}\label{lem:TisFiniteExcursionsAreFinite} Let $P$ be a probability measure over cookie environments satisfying either (ELL), or (WEL) and (IID).
If $\mathbb{P}_0(T_{-1}<\infty)=1$ then all right excursions of the $k$-minimum walk are $\mathbb{P}_0$-a.s.\ finite.
\end{lemma}
\begin{proof}
Fix $m>0$. We will show that the probability that the $m$-th right excursion is infinite is $0$. Let $B(m)$ be the set of all arrow environments such that, if the first $m$ arrows above $0$ are replaced by right arrows, the $k$-minimum walk on the modified arrow environment will never hit $-1$.
Note that an arrow environment on which the $m$-th right excursion  of the $k$-minimum walk is infinite is in $B(m+k-1)$. Let $n=m+k-1$, the lemma will follow once we show that $\PP(B(n))=0$.
Let $C(m)$ be the event that the first $m$ arrows at $0$ are all right arrows. By the assumption of the lemma, $\PP(C(n),B(n))=0$. To conclude we write
\begin{equation}
0=\PP(C(n),B(n))=
\E(\PP_\omega (C(n),B(n)))=
\E(\PP_\omega (C(n))\PP_\omega (B(n))).
\end{equation}
Under (IID) the last term equals $\PP (C(n)) \PP (B(n)))$. By (WEL) $\PP (C(n))>0$ which implies $\PP(B(n))=0$. Under (ELL) $\PP_\omega (B(n))>0$ for a.e.\ $\omega$, which implies that $\PP(B(n))=0$.
%The proof uses finite modification argument which is standard (see \cite{kosygina2008positively} proof of Lemma 8, or \cite{kosygina2012excited} (3.2) and Figure 1.\ there). For convenience we shall supply a sketch. We will prove that the $i$-th right excursion is a.s.\ finite by induction on $i$. For $i=1$ this is true by the assumption for the environment $\om$ after first arrival to $1$. Assume now that the first $i$ right excursions are a.s.\ finite and consider the past including the first step of the $(i+1)$-st excursion. The event that the last excursion is finite depends only on what the walk has done in places $x>0$. Therefore, the probability that the $(i+1)$-st excursion is finite given the past does not change when we modify parts of the past as long as we do not change the parts when $x>0$. In particular it remains the same when we erase all visits to the negative integers and visits to zero are concatenated in time (simply by replacing enough of the first arrows above $0$ to be right arrows). As the modified event has positive probability by (ELL), conditioning on it, the probability for finiteness of the $(i+1)$-st excursion equals the probability that the first excursion is finite conditioned on making a pre-given sequence of first steps on the positive half line. This equals $1$ by the assumption of the Lemma since there is a positive probability to make any pre-given finite sequence of moves.
\end{proof}

\begin{cor}\label{cor:T.infinite.Vs.X.transient} Let $P$ be a probability measure over cookie environments satisfying (ND) and either (ELL), or (IID) and (WEL).
$\PP_0(T_{-1}=\infty ) > 0$ if and only if $\PP_0(X_n\to +\infty) > 0$
\end{cor}

\begin{proof}
For the ``if" implication, note that for $a\in A$, if $t_{-1}=\infty$, then also $X_n\to\infty$ by Remark \ref{rem:koneforallallforone}. Since $P$ satisfies (ND) implies that $\PP(a\in A)=1$ we have $\PP_0(X_n\to +\infty)\ge \PP_0(T_{-1}=\infty)$.
For the ``only if" implication, assume $\mathbb{P}_0(T_{-1}<\infty ) = 1$, then Lemma \ref{lem:TisFiniteExcursionsAreFinite}, $\PP_0$-a.s.\ all right excursions are finite and in particular $\PP_0$-a.s.\ $X_n \nrightarrow +\infty$.
\end{proof}
The proof of the following lemma is similar to the one of Lemma \ref{lem:TisFiniteExcursionsAreFinite} and hence omitted.
\begin{lemma}\label{lem:FinitemodificationForminimumWalk}
Let $P$ be a probability measure over cookie environments satisfying (ND) and either (ELL), or (IID) and (WEL). If $\PP_{0}(\limsup_{n\to\infty}X_n^{(1)}=+\infty)=1$ then $\PP_{0}(\limsup_{n\to\infty}X_n^{(k-\min)}=+\infty)=1$. In particular, in this case the $k$-minimum walk defines a proper $k$-scheduling a.s.\
\end{lemma}

%\begin{proof}
%Let $B(m)$ be the set of arrow environments for which no matter how we change the arrows above $[0,\infty)$, if on the perturbed environment a (single) walker hits $-1$ $m$ times, then after the $m$-th time it will never return to $0$. $B(m)\in \sigma\{a(-1,\cdot),a(-2,\cdot) ,\ldots\}$ (it is independent of the configuration of arrows above $[0,\infty)$.) If the $k$-minimum walk is transient to the left, then there is some minimal $m$ for which $B(m)$ holds.
%Assume by contradiction that $\PP_{\om,0}(\limsup_{n\to\infty}X_n^{(1)}=\infty)<1$. By property (ND) $\PP_{\om,0}(\{X_n^{(k-\min)} \text{is transient to the left}\})>0$,\linebreak and therefore there is some $m>0$ for which $\PP_{\om,0}(B(m))>0$. Let $C(m)$ be the event that the first $m$ arrows in place $0$ are all $0$'s. Note that $C(m)\in\sigma\{a(0,\cdot)\}$.
%Since by definition $\PP_\om$ is a product measure, using (ELL) we have that $\PP_{\om,0}(\text{$a$ is transient to the left})\ge\PP_{\om,0}(B(m),C(m))=\PP_{\om,0}(B(m))\PP_{\om,0}(C(m))>0$, contradicting the assumptions on $\om$.
%\end{proof}

By the Exchangeability Theorem \ref{thm:exchangeability}, we get the following corollary:

\begin{cor}\label{cor:notLeftTransientThenProperScheduling}
Let $P$ be a probability measure over cookie environments satisfying (ND) and either (ELL), or (IID) and (WEL).
If $\AnnealedWalks(\lim_{n\to\infty}X_n^{(i)}=+\infty,\,i=1,...,k)=1$ then $\AnnealedWalks(\limsup_{n\to\infty}X_n^{(k+1)}=+\infty)=1$
\end{cor}
%
%\begin{proof}
%We consider the (random) arrow environment $a$ we get from the cookie environment $\om$.
%Since $a\in A$ a.s., also $a_k\in A$ a.s., and thus a walker walking on it a.s.\ cannot be trapped in a finite interval. Hence, if $a_k$ is not transient to the right or recurrent then it is a.s.\ transient to the left. Since $a$ is a.s.\ $k$-right transient, the $k$-minimum walk $X$ has a minimum point $m\le 0$, so that no particle ever visited the point $m-1$. If $a_k$ is transient to the left let $l>0$ be such that after visiting $m-1$ $l$ times the walker $X^{(k)}$ is always to the left of $m-1$ (see e.g. Lemma \ref{lem:TisFiniteExcursionsAreFinite}).
%We now do a finite modification argument similar to Lemma \ref{lem:TisFiniteExcursionsAreFinite}; change the arrow environment $a$ so that all the first arrows in positions $0$ to $m-1$ are left arrows, and all the first $l$ arrows at $m-1$ are also left arrows. After this modification to $a$, it will be transient to the left (that is, a single particle on $a$ is transient to the left). As in Lemma \ref{lem:TisFiniteExcursionsAreFinite} by ellipticity this means that there is a positive probability that the environment $a$ is transient to the left, contradicting the assumption.
%\end{proof}

\begin{prop}[Inductive directional dichotomy]\label{prop:Inductive directional dichotomy} Let $P$ be a probability measure over the space of cookie environments satisfying (ND) and either (SE) and (ELL), or (IID) and (WEL). Assume that $\mathbb{P} (a \text{ is $k$-right transient})=1$. Then $\mathbb{P}(a_{k} \text{ is recurrent} )=1$ or $\AnnealedCookies(a_{k}\text{ is transient to the right})=1$.
\end{prop}

\begin{proof}
 Let $X'$ be the $(k+1)$-minimum walk, and $T'_{-1}$ the hitting time of $-1$ by $X'$.
 First note by Lemma \ref{lem:FinitemodificationForminimumWalk} the $(k+1)$-min mob walk defines a $(k+1)$-scheduling. Also, Corollary \ref{cor:notLeftTransientThenProperScheduling} implies that $\AnnealedWalks(\limsup_{n\to\infty}X_n^{(k+1)}=+\infty)=1$.

If $\AnnealedWalks$-a.s.\ $T'_{-1}<\infty$ then by Corollary \ref{cor:T.infinite.Vs.X.transient} $\AnnealedWalks(\liminf_{n\to\infty}X'_n=-\infty)=1$, and by exchangeability it means that $\AnnealedWalks(\liminf_{n\to\infty}X^{(k+1)}_n=-\infty)=1$.
Hence $a_k$ is $\AnnealedCookies$-a.s.\ recurrent.

If $T'_{-1} = \infty$ with positive probability then Corollary \ref{cor:CommingBack-k-walkIsTransient} implies that $\AnnealedWalks$-a.s.\ $X'$ is not recurrent. By exchangeability it holds that $a_k$ is a.s.\ not recurrent. Since $\AnnealedWalks$-a.s.\ $\limsup_{n\to\infty}X_n^{(k+1)}=+\infty$ then the only possibility we have left with is that $a_k$ is $\AnnealedCookies$-a.s.\ transient to the right.
\end{proof}

\begin{remark}
Before going on to the proof of Theorem \ref{thm:TransienceThreshold}, let us note that while all the work we have done so far was with regard to right transience, similar analogous statements can be made with regard to left transience. There are two ways to do this. Either repeat all statements and proofs above, this time using left transience, the maximum walk (where we always move the rightmost particle), an analogous version of $Z_n$ on the left half line etc.., or alternatively one may define the reflected arrow environment $\bar a(y,i):=1-a(-y,i)$ and note that ``left"-properties of $a$ (like being transient to the left) are ``right"-properties of $\bar{a}$, then use the statements for right-transience on $\bar{a}$.
\end{remark}

For the proof of Theorem \ref{thm:TransienceThreshold} we shall use the following theorem, which was proved for (SE) and (ELL) in \cite{amir2013zero} and for (IID), (WEL) and (BD) in \cite{kosygina2008positively}.
\begin{theorem}[0-1 law for directional transience]\label{thm:ABO}
Let $P$ be a probability measure over the space of cookie environments satisfying either (SE) and (ELL), or (IID), (WEL) and (BD). Then $\AnnealedCookies(X \text{ is transient to the right})\in\{0,1\}$ and $\AnnealedCookies(X \text{ is transient to the left})\in\{0,1\}$.
\end{theorem}

\begin{proof}[Proof of Theorem \ref{thm:TransienceThreshold}]
If $a$ is recurrent with positive probability, then it follows from Theorem \ref{thm:ABO} that it is in fact a.s.\ recurrent. In particular, $a_k$ is a.s.\ recurrent for all $k$ by definition. In this case the theorem holds with $R=0$. Otherwise, by Theorem \ref{thm:ABO} we may assume \WLOG that $a$ is transient to the right a.s.\ Now, either for all $k$ it holds that $a_k$ is a.s.\ transient to the right, in which case we set $R=+\infty$, or, by By Proposition \ref{prop:Inductive directional dichotomy} there exists some $R<+\infty$ so that $a_k$ is transient to the right a.s.\ if $k\le R$ and recurrent a.s.\ otherwise. The case of left transience follows in an analogous manner by symmetry, with negative $R$.

\end{proof}

We remark that for the proof of Theorem \ref{thm:ABO} in the case of (IID) and (WEL) the assumption (BD) may relaxed to (ND) by noticing that the proof in \cite{amir2013zero} can be adapted using arguments similar to Lemma \ref{lem:TisFiniteExcursionsAreFinite}. This would imply Theorem \ref{thm:TransienceThreshold} holds for (IID), (WEL) and (ND) as well.

\section{Law of Large Numbers for the walkers in the leftover environments}\label{sec:LLN}
In this section we prove the law of large numbers for the walks on the leftover environments (Theorem \ref{thm:ZernerLLN})
To this end we need the notion of directional stationary ergodic environments. This property was defined by Zerner in \cite{zerner2005multi} (there it was used as the definition of being stationary ergodic).
\begin{defi}
A measure $P$ over cookie environments is called \emph{right stationary} (ergodic) if the distribution of $\{\om(x,\cdot)\}_{x\geq 0}$ is stationary (ergodic) with respect to the left shift $\theta^z$, $z\geq 0$, where $\theta^z(\om)(x)=\om(x+z)$. An analogues definition holds for left stationary and ergodic environments.
\end{defi}

Note that if $P$ is (SE) then it is stationary and ergodic in both directions (but not vice versa).

Remember that for an excited random walk $X$, the hitting time $T_x$ is defined by $T_x=\inf\{t\ge 0 : X_t=x\}\in\N\cup\{\infty\}$.
The proof of the law of large numbers in \cite{zerner2005multi} combined with the notation of directional transience gives the following formulation.

\begin{theorem}\label{thm:DLLN}[\cite{zerner2005multi}, Theorem 13]
Assume $P$ be a right stationary and ergodic measure over cookie environments, and assume that $T_x$ is a.s.\ finite for all $x\ge 0$.
Let $$u_+ = \sum_{j\ge 0} \PP (T_{j+1}-T_j>j).$$ Then $\PP$-a.s.\ $\limsup_{n\rightarrow\infty} \frac{X_n}{n} \le \frac{1}{u_+}$. If moreover $u_+<\infty$, then also $\PP$-a.s.\ $\liminf_{n\rightarrow\infty} \frac{X_n}{n} \ge \frac{1}{u_+}$. In particular in this case $\lim_{n\rightarrow\infty} \frac{X_n}{n} = \frac{1}{u_+}$

Symmetrically, whenever $P$ be a left stationary and ergodic measure over cookie environments, so that $T_x$ is a.s.\ finite for all $x\le 0$.
Let $$u_- = \sum_{j\le 0} \PP (T_{-(j+1)}-T_{-j}>j).$$ Then $\PP$-a.s.\ $\liminf_{n\rightarrow\infty} \frac{X_n}{n} \ge \frac{1}{u_-}$. If moreover $u_-<\infty$, then also $\PP$-a.s.\ $\limsup_{n\rightarrow\infty} \frac{X_n}{n} \le \frac{1}{u_-}$. In particular in this case $\lim_{n\rightarrow\infty} \frac{X_n}{n} = \frac{1}{u_-}$.
\end{theorem}

We would like to point out that an immediate consequences of the last theorem is a law of large numbers for the first walker.

\begin{theorem}
Let $P$ be a probability measure over the space of cookie environments satisfying (SE). There is a constant $v$ such that $\frac{X^{(1)}_n}{n}\to v$ a.s.\
\end{theorem}
\begin{proof}
First, all the equalities and inequalities in this proof should be understood ``almost surely". Observe that if P satisfies (SE) then is both right stationary and ergodic and left stationary and ergodic.
We will use Theorem \ref{thm:DLLN} twice, one for the right stationary and ergodic case and one for the left stationary and ergodic one. Using it for right stationarity and ergodicity we get $\limsup_{n\rightarrow\infty} \frac{X_n^{(1)}}{n} \le \frac{1}{u_+^{(1)}}$ and if $u_+^{(1)}<\infty$ then a law of large numbers holds with $\lim_{n\rightarrow\infty} \frac{X_n^{(1)}}{n} = \frac{1}{u_+^{(1)}}$. Using the theorem for left stationarity and ergodicity gives us $\liminf_{n\rightarrow\infty} \frac{X_n^{(1)}}{n} \ge \frac{1}{u_-^{(1)}}$ and if $u_-^{(1)}<\infty$ then a law of large numbers holds with $\lim_{n\rightarrow\infty} \frac{X_n^{(1)}}{n} = \frac{1}{u_-^{(1)}}$. If both $u_+^{(1)}=\infty$ and $u_-^{(1)}=\infty$ then we get $\limsup_{n\rightarrow\infty} \frac{X_n^{(1)}}{n} \le 0 \le \liminf_{n\rightarrow\infty} \frac{X_n^{(1)}}{n}$ and so
$\lim_{n\to\infty}\frac{X_n^{(1)}}{n} = 0$ \end{proof}

In Section $5$ of \cite{zerner2005multi}, Zerner observed the following:
\begin{lemma}\label{lem:directionalSEleftover}
If $P$ is a right (left) stationary and ergodic measure over cookie environments which is a.s.\ right (left) transient, then the distribution of the leftover environment is also right (left) stationary and ergodic.
\end{lemma}

We are now ready to prove the law of large numbers - Theorem \ref{thm:ZernerLLN}
\begin{proof}[Proof of Theorem \ref{thm:ZernerLLN}]
Let $R$ be the threshold from Theorem \ref{thm:TransienceThreshold}. Assume without loss of generality that $R\geq 0$.
Let $X^{(k)}$ be the walk on the $(k-1)$-leftover environment.
Then, for any $k\leq R$, $X^{(k)}$ is a.s.\ transient to the right, and in particular $\liminf_{n\rightarrow\infty} \frac{X_n^{(k)}}{n} \ge 0$. Moreover, by inductively applying Lemma \ref{lem:directionalSEleftover} we get that $a_j$, $j\leq R$, are all right stationary and ergodic. Therefore by Theorem \ref{thm:DLLN} $\limsup_{n\rightarrow\infty} \frac{X_n^{(j)}}{n} \le \frac{1}{u_+^{(j)}}$ for $j\leq R$ and if $u_+^{(j)}<\infty$ then a law of large numbers holds with $\lim_{n\rightarrow\infty} \frac{X_n^{(j)}}{n} = \frac{1}{u_+^{(j)}}$. On the other hand if $u_+^{(j)}=\infty$ we may use both inequalities to get that $\lim_{n\to\infty}\frac{X_n^{(j)}}{n} = 0$.

Set $k=\lfloor R+1\rfloor$. To finish the proof it is enough to show that a law of large numbers holds for $X^{(k)}$, which is a.s.\ recurrent. In other words, it is sufficient to show that $\lim_{n\rightarrow \infty} \frac{X^{(k)}_n}{n}=0$. Indeed, by the definition of the leftover environments, the environment left over by the walk $X^{(k)}$ is the balanced environment (that is, there are no cookies left), thus all subsequent walks will be simple random walks and in particular satisfy a law of large numbers with speed $0$.
By an inductive use of Lemma \ref{lem:directionalSEleftover} it follows that the distribution of $a_k$ is right stationary and ergodic. Now, $u_+^{(k)}=\infty$ as otherwise by Theorem \ref{thm:DLLN} we would have a positive speed, and in particular right transience. Therefore by Theorem \ref{thm:DLLN} $\limsup_{n\rightarrow\infty} \frac{X_n^{(k)}}{n} \le \frac{1}{u_+^{(k)}}=0$. To finish the proof we need to show that $\liminf_{n\rightarrow\infty} \frac{X_n^{(k)}}{n} \ge 0$. To that end let us define a random variable $M$ to be the minimal position of all first $k$ walkers: $$M=\min\{X^{(j)}_n:j=1,...,k-1,\, n\ge0\}.$$ (In the case that $k=1$ define $M=0$). Since all $k$ walkers are a.s.\ transient to the right then $M$ is a.s.\ finite. Let $T_{M}$ be the first hitting time of $M$ by the $k$-th walker. As $X^{(k)}$ is a.s.\ recurrent $T_M$ is a.s.\ finite. Now define a new process $X'$ by $X'_n=X^{(k)}_{n+T_M}$. Note that $X'$ is an excited random walker in the environment $\om'$ so that $\om'(x,i)=\om(x+M,i)$ for all $x\le 0$. In particular the distribution of $w'$ is left stationary and ergodic. It follows from Theorem \ref{thm:DLLN} that $\liminf_{n\to\infty}\frac{X'_n}{n}\ge 0$ (indeed, the involved $u_-$ must be $\infty$ otherwise the walker $X^{(1)}$ on $\omega$ is transient to the left, contradicting Theorem \ref{thm:TransienceThreshold} or the assumption $R\ge 0$). But as $T_M$ is a.s.\ finite, writing $m=n-T_M$ for large $m$ one gets that
$$0\le \liminf_{n\to\infty}\frac{X'_n}{n} =\liminf_{m\to\infty}\frac{X^{(k)}_{m}}{m-T_M} =\liminf_{m\to\infty}\frac{X^{(k)}_{m}}{m}.$$
Thus $\lim_{n\to\infty}\frac{X^{(k)}_{n}}{n}=0$ as required, concluding the proof.
\end{proof}

\section{Branching processes with migration }\label{sec:BPwM}
%%%%%%%%%%%%%%%%%%%%%%%%%%%
As mentioned above, Kosygina and Zerner considered in \cite{kosygina2008positively} the case of (IID), (BD), and (WEL). A crucial observation in their paper is that the process $Z$ in this case has a particular form of a branching process with migration (BPwM). As the field of BPwM is quite developed since the 1970's, they looked for a theorem on BPwM in the desired form. However as no theorem with an accurate formulation was found in the literature, they instead used theorems of Formanov-Yasin \cite{FY1989BPwMig}, and Formanov-Yasin-Kaverin \cite{FYK1990BPwMig} which had in some sense the ``closest" formulation. One of the main steps in \cite{kosygina2008positively} was to deduce a theorem of `their' form of BPwM from the above mentioned theorems. Although the argument in Kosygina and Zerner \cite{kosygina2008positively} did not follow the lines of first explicitly formalizing a theorem regarding some general BPwM and then using it, it is however possible to reformulate their argument in such a way. Since the form of Kosygina and Zerner is also convenient for us here, we will formulate a weaker version of a theorem which is implicit in their paper (up to minor change of parameters), and whose prove can be deduced directly by following their argument (see the discussion in the end of Section 3 of \cite{kosygina2008positively}).

Fix $M\in\N$ and let $\nu_1,...,\nu_M$ be probability distributions on $\Z\cap [-M,+\infty)$ so that $\nu_1(i),...,\nu_M(i)>0$ for all $i\ge 0$, and the cumulative distribution functions satisfy $\nu_j((-M,x))\ge \nu_M((-M,x)))$ for all $1\le j\le M$ and all $x\in \R$.
Denote by $\gamma$ the expectation of $\nu_M$, which we assume to be finite.
\begin{equation}\label{eq:gamma}
\gamma:=\sum_{j\ge 0} j\nu_M({j})<\infty.
\end{equation}
 Let $\xi^{j}_i,i,j\ge 0$ be i.i.d.\ $\textmd{Geom}(\frac{1}{2})$ random variables, and for each $i=1,...,M$ let $\eta^{j}_i,j\ge 0$ be i.i.d.\ random variables with distribution $\nu_i$.

 \begin{defi}
 For a discrete time process $Y$ on $\Z_+$, its \emph{total progeny} is defined by $$\tilde{Y}=\sum_{n=0}^\tau Y_n,$$ where $\tau=\inf\{n\ge 0:Y_n=0\}$ is the (perhaps infinite) \emph{hitting time} of $0$ by $Y$.
    \end{defi}

\begin{theorem}\label{thm:KZBPwM} Fix initial value $y\in\Z_+$ and a constant migration $N\in\Z$.
Let $Y=(Y_n)_{n\ge 0}$ be a process defined by $Y_0=y$, and
\begin{equation}\label{eq:BPwMformY}
Y_{n+1}=\sum_{i=1}^{Y_n+N-M}\xi_i^{(n+1)} + \eta_{(Y_n+N)\wedge M}^{(n+1)}
\end{equation}
where by convention $\sum_{i=1}^{-j}\xi_i^{(n+1)}\equiv 0$ and $\eta_{-j}^{(n+1)}\equiv 0$, $j \ge 0$. Then $Y$ dies out (that is $\tau<\infty$ or equivalently $\tilde{Y}<\infty$) a.s.\ if and only if $\gamma -M + N \le 1$. Moreover, the total progeny $\tilde{Y}$ of $Y$ has finite expectation if and only if $\gamma -M + N < - 1$.
\end{theorem}

\begin{remark}
Kosygina and Zerner proved Theorem \ref{thm:KZBPwM} for the case in which $\nu_j,$ $1\le j\le M$, have specific distributions, $N=0$ in the first part, and $N=1$ in the ``moreover" part, but their proof works also in this formulation. Indeed, they derived their result on processes of the form \eqref{eq:BPwMformY} from results in the literature for what is known as $(\mu,\nu)$ - branching processes (see \cite{kosygina2008positively} for the definition and the main result from the literature - Theorem A in their paper). One of their main steps in the paper is to show how to move from the formulation \eqref{eq:BPwMformY} to $(\mu,\nu)$ - branching processes, this is done in chapter 4 of their paper. We sketch the argument in our setting.
The first step is to define $(Y_n')$ by $Y'_0=1$ and $Y'_{n+1} = \sum_{i=1}^{Y'_n+N-M} \xi^{n+1}_i  +  \eta_M^{(n+1)}$. Since $\nu_1(i),...,\nu_M(i)>0$ for all $i\in\N$ and as the transition probability from $i$ to $j$ in both processes $Y$ and $Y'$ differ only for $i\in\{0,...,M-1\}$ then if one of the processes goes to infinity w.p.p the so does the other (see the discussion above Lemma 6. of \cite{kosygina2008positively}). Moreover, Lemma $15$ of \cite{kosygina2008positively} together with the conditions on $\nu_j$ implies that the total progeny of $Y'$ has finite expectation if and only if so does $Y$.
The last step is to connect the process $Y'$ to a $(\mu,\nu)$ - branching processes, call it $Z$. For that we define inductively $Z_n = Y'_{n+1}-\eta_M^{(n)}$. Then similarly to Lemma 6 of \cite{kosygina2008positively} we get that $Z$ is a $(\mu,\nu)$ - branching processes, where $\mu=Geom(\frac{1}{2})$ and $\nu$ is the common distribution of $\eta_1^{(1)}+N-M$. Theorem \ref{thm:KZBPwM} now follows from Theorem (A) in \cite{kosygina2008positively}.
\end{remark}

\section{Right transience of the walkers in leftover environments}\label{sec:transienceLeftover}

In this chapter we show that under the assumptions of Theorem \ref{thm:bddCaseTransience} the process $Z$ is a BPwM, and then derive its proof using Theorem \ref{thm:KZBPwM}. In the next lemma we will show that the process $Z$ defined in \eqref{eq:defOfz} under $\QuenchedTrees$ has the same distribution as a process defined in Theorem \ref{thm:KZBPwM}.
Let $B_i,i\ge 1$, be a sequence of independent Bernoulli random variables so that $B_i\sim B(p_i), 1\le i\le M$ and $B_i\sim B(\frac{1}{2})$ for $i>M$.
Define for each $1\le j\le M$ and $r\ge 0$
\begin{equation}\label{eq:nuJ}
f_j(r)=\AnnealedCookies(\inf\{t\ge 1: \sum_{i=1}^t (1 - B_i)  = j\}-j=r).
\end{equation}

\begin{lemma}\label{lem:ZisBPwM} Let $Z$ be the process defined in \eqref{eq:defOfz}. Let $(Y,P)$ be from Theorem \ref{thm:KZBPwM} with parameters $N=-(k-1)$, $y=k$ and $\nu_j=f_j(\cdot)$, $1\le j\le M$.
Then $(Y,P)$ and $(Z,\PP^1)$ have the same distribution.
\end{lemma}
\begin{proof} First observe that since $P$ is (IID) and (ND) then $Z$ is a Markov chain on the non-negative integers. Now, the proof follows by induction on $n\ge 0$. For $n=0$, both processes have a.s.\ the same initial conditions. For the induction step, it is enough to show that for any $m$, the distribution of $Z_{n+1}$ given $Z_n=m$ under $\PP^1$ is the same as that of $Y_{n+1}$ given $Y_n=m$ under $P$. Assume therefore that $Z_n=Y_n=m$. Then, under $\PP^1$, $Z_{n+1}$ equals the number of $1$'s in $\om(n,\cdot)$ prior to $m+N$ $0$'s. If $m+N\le M$ then, under $\PP^1$, $Z_{n+1}$ has a distribution $\nu_{m+N}$, which coincides with $Y_{n+1}$ under $P$. If $m+N=M+j$, with $j>0$, then $Z_{n+1}$ is distributed as the number of $1$'s in $\om(n,\cdot)$ prior to $M$ $0$'s plus a negative binomial random variable $Q\sim NB(\frac{1}{2},j)$. The first summand is distributed as $\nu_M$, while a negative binomial random variable is a sum of i.i.d.\ geometric random variables. Therefore, also in this case the distribution of $Z_{n+1}$ under $\PP^1$ coincides with that of $Y_{n+1}$ under $P$.
\end{proof}

\begin{lemma}[Basdevant-Singh \cite{basdevant2008speed}]\label{lem:valueOfGamma}
Let $Y$ be the process defined in Lemma \ref{lem:ZisBPwM}, and $\gamma$ be the related expectation defined in \eqref{eq:gamma}. Then $\gamma = \delta + M$.
\end{lemma}

\begin{proof}
We follow Lemma 3.3 of \cite{basdevant2008speed} and equation (23) of \cite{kosygina2008positively}.
Let $F$ be the number of failures in $M$ trials, and given $F$ define $H$ to be a negative binomial random variable $H\sim NB(\frac{1}{2},F)$. Then $M-F$ is the number of successes in $M$ trials, and $\nu_M\sim M - F + H$. Therefore,
$\gamma= M-\mathbb{E}[F]+ \mathbb{E}[H]=M-\sum_{i=1}^M(1-p_i)+\sum_{i=1}^M p_i = \sum_{i=1}^M 2 p_i = \delta + M$.
\end{proof}

We get the following corollary.
\begin{cor}\label{cor:deltaVs.T-1}
Let $P$ be a probability measure over cookie environments satisfying (IID), (WEL), and (BD). Let $T_{-1}$ be as usual the hitting time of $-1$ by the $k$-minimum walk $X$. Then $\AnnealedWalks(T_{-1}=\infty)>0$ if and only if $\delta>k$
\end{cor}
\begin{proof}
Let $(Y,P)$ be from Theorem \ref{thm:KZBPwM} with parameters $N=-(k-1)$, $y=k$ and $\nu_j=f_j(\cdot)$, $1\le j\le M$.
By Lemma \ref{lem:ZisBPwM} $(Y,P)$ and $(Z,\PP^1)$ have the same distribution, and by Theorem \ref{thm:KZBPwM} $(Y,P)$ has a positive probability of survival if and only if $\gamma < M+1-N=M+k$. Lemma \ref{lem:valueOfGamma} tells us that $\gamma = \delta + M$. Put together we get that $(Z,\PP^1)$ has a positive chance of survival if and only if $\delta>k$.
The Corollary now follows from Lemma \ref{lem:T-1iffZ}, which implies that $\AnnealedWalks(T_{-1}=\infty >0)$ if and only if $\PP^1(Z_n>0 \text{ for all } n) > 0$.
\end{proof}
We are now able to prove Theorem \ref{thm:bddCaseTransience}.

\begin{proof} Assume \WLOG~that $\delta\ge 0$.
By Corollary \ref{cor:deltaVs.T-1} $\delta>k$ if and only if $\AnnealedWalks(T_{-1}=\infty)>0$. By Corollary \ref{cor:T.infinite.Vs.X.transient}, the latter holds if and only if $\AnnealedWalks(X_n\to\infty)>0$, which holds, by Corollary \ref{cor:seqRightTranEquivMinRightTan} if and only if $\PP(a \text{ is $k$-right transient})>0$, and by Theorem \ref{thm:TransienceThreshold} is equivalent to $\AnnealedWalks(X_n\to\infty)=1$. Since by Corollary \ref{cor:seqRightTranEquivMinRightTan} the $k$-minimum walk is a.s.\ transient to the right if and only if a.s.\ all the walks $X^{(i)}$ on the $i$-leftover environment, $i<k$, are transient to the right, we are done.
\end{proof}

\section{Positive speed for the $k$-minimum walk}\label{sec:SpeedForkmin}
Throughout this section we assume that $P$ is (IID), (BD) and (WEL) probability measure over cookie environments, and \WLOG, that the transience threshold $R$ from Theorem \ref{thm:TransienceThreshold} is nonnegative. By Theorem \ref{thm:TransienceThreshold} $a$ is $\PP$-a.s.~$R$-right transient (but not $(R+1)$-right transient).
Note that for $k>R$ the walks $X^{(k)}$ are all $\AnnealedWalks$-a.s.~recurrent and therefore satisfy a law of large numbers with speed $0$. In particular we may assume that $R>0$. Fix $1\leq k\leq R$ and let $X=X^{(\min)}$ be the $k$-minimum walk. Since $a$ is $\PP$-a.s.~$k$-right transient, Lemma \ref{le:inf_opt_reg} implies that there $\AnnealedWalks$-a.s.~exists an infinite sequence $0\leq r_1<r_2<\ldots$ of regeneration positions for $X$.

\begin{remark}\label{re:reg_exchangability}
Consider the $k$-min mob walk. One may equivalently describe the regeneration positions as the set of visited positions $r\ge 0$ from which no particle jumps to the left. It follows from the last definition together with the exchangeability property of the particles (Theorem \ref{thm:exchangeability}), that the set of regeneration positions is independent of the $k$-scheduling chosen.
\end{remark}

Let $r_i$ be the $i$th nonnegative regeneration position and $\tau_i=T_{r_i}=\inf\{t\ge 0: X_t=r_i\}$ be the corresponding regeneration time.
Since we assumed the measure over the cookie environments is (IID) then the sequence
$(r_1,\tau_1),(r_{k+1}-r_k,\tau_{k+1}-\tau_k)\
(k\ge 1)$ of random vectors is independent under $\AnnealedWalks$.
Furthermore, the random vectors
$(r_{k+1}-r_k,\tau_{k+1}-\tau_k),\ k\ge 1$, have the
same distribution under $\AnnealedWalks$.

It follows from the renewal theorem (see e.g.~\cite{grimmett2001probability} Section $10.5$) that
\begin{equation}\label{eq:ren}
\E_0[r_2-r_1]=\AnnealedWalks[r_1=0]^{-1}<\infty.
\end{equation}
Moreover, the ordinary strong law of large numbers implies that
\begin{equation}\label{eq:vmin}
\lim_{n\to\infty}\frac{X_n}{n}=\frac{\E_0[r_2 - r_1]}
{\E_0[\tau_2-\tau_1]}=:v^{(k-\min)}\qquad\mbox{$\QuenchedCookies$-a.s.,}
\end{equation}

We are ready to present the main result of this section:

\begin{theorem} \label{thm:PosSpeed}
$v^{(k-\min)}>0$ if and only if $\delta > k+1$
\end{theorem}

By \eqref{eq:ren} and \eqref{eq:vmin} we have the following
\begin{equation}\label{eq:first}
\mbox{$v^{(k-\min)}>0$\quad if and only if \quad $\E_0[\tau_2-\tau_1]<\infty.$}
\end{equation}
Thus analyzing the positivity of the speed of the $k$-minimum walk boils down to understanding when is $\E_0[\tau_2-\tau_1]<\infty.$
To do so, we will follow closely the proof strategy of \cite{kosygina2008positively} and compare the $\E[\tau_2-\tau_1]$ with the total progeny of a branching process with migration.

\begin{defi}
A step $t$ of a $k$-mob walk is called a downcrossing of the edge $\{n-1,n\}$ if at time $t$ one of the particles moves from $n$ to $n-1$. Similarly, a step is called an upcrossing of $\{n-1,n\}$ if at time $t$ one of the particles moves from $n-1$ to $n$.
\end{defi}
It is clear that every step of a $k$-mob walk is either a downcrossing or an upcrossing.
By exchangeability, the total number of downcrossings of each edge is independent of the chosen scheduling. Note that for the $k$-min mob walk, a step $t$ is a downcrossing if and only if the $k$-minimum walk satisfies $X_{t+1}<X_t$.

For $n\ge 0$ we introduce
\begin{equation}\label{eq:Uk}
D_n:=\#\left\{t\mid \tau_1<t<\tau_2,\ X_t = x_2 -n,
\ X_{t+1}=x_2 - n - 1\right\}
\end{equation}
to be the numbers of downcrossings of the edge $(x_2-n,x_2-n-1)$ between
the times $\tau_1$ and $\tau_2$ by the $k$-min mob walk.

\begin{lemma}\label{lem:mom} Let $m\ge 1$. The $m$-th moment of $\tau_2-\tau_1$ under
$\QuenchedCookies$
is finite if and only if the $m$-th moment of $\sum_{n\ge 1}D_n$ is
finite.
\end{lemma}

\begin{proof}
The number of upcrossings between $\tau_1$ and $\tau_2$ is
$k \cdot (x_2-x_1) + \sum_{n\ge 1}D_n$, since
the $k$ particles need to move from $x_1$ to $x_2$ and since each downcrossing needs to be
balanced by an upcrossing. Each step is either an
upcrossing or a downcrossing, therefore,
\begin{equation}\label{eq:kpick}
\tau_2-\tau_1 = k\cdot (x_2 - x_1) + 2\sum_{n\ge 1}D_n.
\end{equation}
For every $n\in\{x_1+1,\ldots,x_2-1\}$ $D_n\ge 1$, otherwise $n$ would be
another regeneration position. Hence, $x_2-x_1\le 1+ \sum_{n\ge 1}D_n$ and, by \eqref{eq:kpick},
\[2\sum_{n\ge 1}D_n\le \tau_2-\tau_1\le k+ (2+k) \sum_{n\ge 1}D_n.\]
This implies the claim.
\end{proof}

Next we will show that $D_n$ is a BPwM of the form \eqref{eq:BPwMformY}. As before, let $B_i,i\ge 1$, be a sequence of independent Bernoulli random variables so that $B_i\sim B(p_i), 1\le i\le M$ and $B_i\sim B(\frac{1}{2})$ for $i>M$.
Define
\begin{equation}\label{eq:nuJ2}
g_j(r)=\PP(\inf\{t\ge 1: \sum_{i=1}^t B_i = j\}-j=r)
\end{equation}
to be probability that the number of 0's prior to the first $j$ 1's equals $r$.
\begin{lemma}\label{lem:DandZhaveSameDistr}
The distribution of $D_n$ under $\AnnealedWalks$ is the same as the distribution of $Y_n$ under $\QuenchedCookies$, where $Y$ is defined to be a BPwM as in Theorem \ref{thm:KZBPwM} with $\nu_j\sim g_j(\cdot)$, $1\le j\le M$ and $N=M+k$.
\end{lemma}
The proof is similar to the one of Lemma \ref{lem:ZisBPwM} and hence omitted.

\begin{lemma}\label{lem:gamma'}
$\gamma':=\sum_{r=o}^\infty r g_M(r)=M-\delta$.
\end{lemma}
\begin{proof}
Exchange $p_i$ with $1-p_i$, $1\le i\le M$ and use Lemma \ref{lem:valueOfGamma} to get
$\gamma'=\sum_{r=o}^\infty r g_M(r)=2\sum_{i=1}^M(1-p_1)=M-\delta$
\end{proof}

We can now prove Theorem \ref{thm:PosSpeed}.

\begin{proof}[Proof of Theorem \ref{thm:PosSpeed}]
By \eqref{eq:first} it is enough to show that $\delta>k+1$ if and only if $E_0[\tau_2-\tau_1]<\infty$. By Lemma \ref{lem:mom}, the latter holds if and only if $\sum_{n\ge 1}D_n$ has a finite first moment. By Lemma \ref{lem:DandZhaveSameDistr} the latter holds if and only is the first moment of the total progeny of the process $Y$, defined in Lemma \ref{lem:DandZhaveSameDistr}, is finite. By Theorem \ref{thm:KZBPwM} this holds if and only if $\gamma' -M + N < - 1$, and by Lemma \ref{lem:gamma'} the latter holds if and only if $\delta>k+1$.
\end{proof}

\section{Speed for walkers in the leftover environments}\label{sec:SpeedLeftover}
In this section we shall prove Theorem \ref{thm:bddCaseSpeed}.
To do so we will give a formula relating the speeds of the walks on the left-over environments to the speeds of the $k$-minimum walks.
We will assume, wlog, that $\delta \geq 0$.
\begin{proof}[Proof of Theorem \ref{thm:bddCaseSpeed}]
Since we deal with several scheduling in this section, we shall denote by $X^{k-\min}$ the $k$-minimum walk on $a$ and by $X^{(i)}$ the walk on the $(i-1)$-leftover environment. By equation \eqref{eq:vmin} and Theorem \ref{thm:DLLN} these walks satisfy a law of large numbers, and we denote their speeds by $v^{k-\min}$ and $v_i$ respectively.

Fix some $k<\delta$.
 Let $\{r_n\}$ be the set of regeneration positions for the $k$-minimum walk, and $\tau^{(k-\min)}_n$ their hitting time by the walk. Remark \ref{re:reg_exchangability} tells us that the regeneration positions are independent of the chosen scheduling. This implies that $\{r_n\}$ are also regeneration positions for $X^{(i)}$ for all $1\leq i\leq k$ (though there may be other regeneration positions as well). Denote by $\tau^{(i)}_n$ the hitting time of $r_n$ by $X^{(i)}$.
Since these walks satisfy a law of large numbers it follows that $\PP$-a.s.\
 \begin{equation}\label{eq:vseqexist} v_i = \lim_{n\rightarrow\infty}\frac{r_n}{\tau^{(i)}_n} \ \text{ and } \  v^{(i-\min)} = \lim_{n\rightarrow\infty}\frac{r_n}{\tau^{(i-\min)}_n}\  \end{equation}
Given any proper $k$-scheduling $S$, define $L^{(S)}((-\infty,n])=\sum_{y=-\infty}^n L^{(S)}(y)$ - the total number of steps by particles in the interval $(-\infty,n]$. By the
 Exchangeability Theorem \ref{thm:exchangeability}, $L^{(S)}$ is independent of the choice of $S$, and may thus be denoted simply $L((-\infty,n])$.\footnote{In fact, one may use this to define the speed of a general $k$-mob walk on any $k$-right transient environment simply as $\lim_{n\rightarrow\infty}\frac{n}{L((-\infty,n])}$. However even for $1$ particle this limit may exist even when the regular speed does not}
 Comparing the sequential and minimum walks on $k$ particles, we get that for any regeneration position $r_n$
 \begin{equation}\label{eq:seq-min}
 L((-\infty,r_n]) = \tau^{(k-\min)}_n=\sum_{i=1}^k \tau_n^{(i)} \ \text{ and  } \tau^{(i)}_n=\tau^{(i-\min)}_n-\tau^{((i-1)-\min)}_n.
 \end{equation}
Since $\tau^{(i)}_n\ge r_n$ for any $i$, the second equality together with \eqref{eq:vseqexist} give that $v^{(i-\min)} \leq \frac{v^{((i-1)-\min)}}{v^{((i-1)-\min)}+1}$ and in particular the lsequence of speeds of the $i$-minimum walk is strictly decreasing in $i$ until it zeroes out.
Dividing $\tau^{(k-\min)}_n$ by $r_n$ and taking limits \eqref{eq:seq-min} gives that $v^{(k-\min)}>0$ if and only if $v_i>0$ for all $1\leq i\leq k$, and that in this case
\begin{equation*}%\label{eq:k-mobSpeedFormula}
\frac{1}{v^{(k-\min)}} = \sum_{i=1}^k \frac{1}{v_i}
\end{equation*}
By Theorem \ref{thm:PosSpeed} $v^{(k-\min)}>0$ for $k<\delta-1$ and $v^{(\lfloor \delta-1\rfloor -\min)}=0$.
It follows by induction that $v_1=v^{(1-\min)}$, $v_{\lfloor \delta-1\rfloor}=0$ and for $2\leq i<\delta-1$
$$v_i = \frac{1}{\frac{1}{v^{(i-\min)}}-\frac{1}{v^{((i-1)-\min)}}} = \frac{ v^{(i-\min)}v^{((i-1)-\min)}} {v^{((i-1)-\min)}-v^{(i-\min)}}>0$$
With the last inequality following from the strict monotonicity of $v^{(i-\min)}$.
For $i> \delta-1$ we have $v_i=0$ by Theorem \ref{thm:bddCaseTransience} that the $(i-1)$-st leftover environment is a.s.\ recurrent.
This completes the proof of the theorem.
\end{proof}

%\begin{proof}[Proof of Theorem \ref{thm:bddCaseSpeed}]
%By Theorem \ref{thm:PosSpeed}, $v^{(k-\min)}>0$ if and only if $k<\delta+1$, and thus (using strict monotonicity of $v^{(k-\min)}$ in $k$) we get that $v_k>0$ for all $k<\delta-1$ and $v_k=0$ for $k=\ulcorner \delta-1 \urcorner$.
%
%\end{proof}

\section{Stationary leftover environments}\label{sec:stationary}
As mentioned in the introduction, the leftover environments are not necessarily stationary even when the original environment is i.i.d.~since there is a special point - $0$ - at which the walker generating this environment started. Instead, the leftover environments inherit directional stationarity, as discussed in section \ref{sec:LLN}. In a private communication, Jonathon Peterson asked whether it is possible to introduce a ``stationary version" of the leftover environment such that the results of this paper would carry over to these environments.
In this section we describe a way to redefine leftover environments so that they inherit the (SE) property of the original environment.
To this end, we consider walks in which the walkers start in some general set of locations $x_1,\ldots,x_k\in \Z$ (in contrast with the rest of the paper where we assumed, mainly for ease of notation, that the walkers all started at $0$).
The following is a generalization of $k$-right transience, to include also initial positions:
\begin{defi}
We say that an arrow environment is $k$-right transient w.r.t.\ initial positions $x_1,\ldots, x_k$ if for a $k$-mob walk on $a$ with these initial positions all $k$ particles go to $+\infty$. We call an environment strongly $k$-right transient if it is transient to the right w.r.t.\ all $k$-tuples. These definitions go over to probability measures $\mu$ over cookie environments in the usual way, that is whenever they occur $\mu$-a.s.\
\end{defi}
Note that by the Exchangeability Lemma \ref{lem:exchangeabilty}, the choice of proper scheduling of the $k$-mob walk does not change the $k$-right transience property, nor its local time.
We can now define the leftover environment left by $k$ walkers with given initial positions:
\begin{defi}
Let $a$ be $k$-right transient w.r.t.\ some $x_1,\ldots,x_k\in \Z$. Let $L^{x_1,\ldots,x_k}(\cdot)$ denote the local time of some (any) $k$-mob walk on $a$ starting at initial positions $x_1,\ldots,x_k$. We define the leftover environment of $a$ generated by $k$ walkers starting at $x_1,\ldots , x_k$ by $$LO^{x_1,\ldots,x_k}(a)(y,i) := a(y,i+L^{x_1,\ldots,x_k}(y)).$$
\end{defi}

%To define a stationary version of the leftover environment, we would like to take the starting point of the walker to $-\infty$.
%To do this, we will have to assume that the original environment is transient to the right. In fact, we will use a slightly stronger version of right-transience for this section - we say that an arrow environment $a$  is transient to the right if the walker starting at $m$ will go to infinity for any choice of point $m\in \Z$. Note that for stationary measures over cookie environments, right transience of the measure implies a.s.\ right transience of the corresponding arrow environment.

To define a stationary version of the leftover environment, we would like to take the starting points of the walkers to $-\infty$.
To do so, we need the following monotonicity lemma.

\begin{lemma}[Local time monotonicity in initial positions]\label{le:initial positions}
Let $a$ be a non-degenerate $k$-right transient arrow environment. Let $x_1,\ldots,x_k,y_1,\ldots,y_k\in\Z$ such that $x_i\leq y_i$ for al $1\leq i\leq k$. Let $X$ be some $k$-mob walk on the arrow environment $a$~with walkers at initial positions $x_1,\ldots,x_k$, let $Y$ be a $k$-mob walk on the arrow environment $a$~with walkers at initial positions $y_1,\ldots,y_k$, and denote their asymptotic local times by $L^X$ and $L^Y$ respectively. The following inequality holds: $$L^Y(z)\leq L^X(z)$$ for all $z\in \Z$. That is, moving the initial positions of the walkers to the right cannot increase local times of $k$-right transient arrow environments.
\end{lemma}
\begin{proof}
Using induction and by shifting the arrow environment it is enough to prove the lemma for the case $x_i=y_i$ for $1\leq i\leq k-1$ and $x_k=0, y_k=1$.
As the local time for $X$ is everywhere finite, we may use the Exchangeability Lemma \ref{lem:exchangeabilty} to choose our favourite scheduling (with the initial positions $x_1,\ldots,x_k$). Thus we may ``first" send walkers $1$ to $k-1$ to infinity, and then deal with the $k$-th walker. The word ``first" is in parenthesis as we cannot really send them off to infinity as this would not be a proper scheduling, but we may essentially do so as in Lemma \ref{le:sequencial}, without this change effecting the path of the last walker. We are now left with a new ``leftover" arrow environment $a_1= LO^{x_1,\dots,x_{k-1}}(a)$, and to compare the local times of $X$ and $Y$ we need only to compare the local times of the walk starting from $0$ on $a_1$ with those of the walk started from $1$. We will refer to these two walks as $X^0$ and $X^1$. $k$-right transience of $a$ implies $X^0(n) \rightarrow \infty$ as $n\to\infty$.
Define the following sequences of times from $X^0$:
$l_0=0$ and for $i\geq 1$ let $$r_i=\inf \{n\geq l_{i-1}: \ X^0(r_i)=1\} \ \text{and} \ l_i=\inf \{n\geq r_i: \ X^0(l_i)=0\}$$
Let $L_i = \{X^0(t)\}_{{l_i} \leq t< r_{i+1}}$ be the $i$-th left-excursion of $X^0$ and similarly $R_i= \{X^0(t)\}_{r_i \leq t< l_i}$ - the $i$-th right-excursion of $X^0$. By right transience of $X^0$ there is some $I$ for which so that the $I$-th right excursion never ends, that is $r_I< \infty$ but $l_I=\infty$. The path of $X^0$ is just the concatenation \begin{equation}\label{eq:x0path} \{X^0\}=L_0R_1L_1R_2L_2\ldots L_{I-1}R_I. \end{equation}
Note that the left paths $L_i$ depend only on $(a(x,\cdot), x\le 0)$, the arrows on the non-positive integers. Symmetrically the right paths $R_i$ depend only on $(a(x,\cdot), x\ge 1)$, the arrows on the strictly positive integers. Hence the walk $X^1$ can be written as
\begin{equation}\label{eq:x1path}\{X^1\}=R_0L_0R_1L_1,...L_{I-1}R_I.\end{equation} (If $I=0$ then $X^0=L_0R_I$ and $X^0=R_0$.) Hence, the walker $X^1$ started at $1$ will not make the last left path $L_I$. In particular, we get the desired inequality for the local times..

\end{proof}
We get the following corollary for (SE) measures:
\begin{cor}\label{cor:strongtransience}
Let $P$ be probability measure over cookie environments satisfying (SE). $\PP$ is $k$-right transient w.r.t.\ some $x_1,\ldots , x_k$ if and only if it is strongly $k$-right transient.
\end{cor}
\begin{proof}
Let $y_1,\dots,y_k,x_1,...,x_k\in \Z$. We want to show that if $\PP$ is $k$-right transient w.r.t.\ $x_1,...,x_k$ then it is also $k$-right transient w.r.t.\ $y_1,...,y_k$. Let $m=\max\{x_1,\ldots,x_k\}-\min\{y_1\ldots,y_k\}$, and consider a $k$-mob walk on $a$ with initial positions $y_1+m,\ldots,y_k+m$. By stationarity of $\PP$ the probability that $a$ is $k$-right transient w.r.t.\ $y_1+m,\ldots,y_k+m$ is equal to the probability that $a$ is $k$-right transient w.r.t.\ $y_1,\ldots,y_k$. As $y_i+m\geq x_i$ for all $i$, Lemma \ref{le:initial positions} ensures that if $a$ is $k$-right transient w.r.t.\ $x_1,\ldots,x_k$ then it is also $k$-right transient w.r.t.\ $y_1+m,\ldots,y_k+m$. Hence $a$ is $\PP$-a.s.\ $k$-right transient w.r.t.\ $y_1,\ldots,y_k$. As $y_1,\ldots,y_k$ were arbitrary, this concludes the proof.
\end{proof}
%
%The main observation is the following:
%\begin{lemma}
%Let $a$ be a transient to the right arrow environment. Then for any $m<n$ $LO^m(a)(x,i) = LO^n(x,i)$ for all $x\geq n$ and $i\geq 1$.
%That is the leftover environment to the right of $n$ remains the same for any choice of starting point to the left of $n$.
%\end{lemma}
%\begin{proof} This is a direct consequence of the fact that every time the walker on $a$ is to the left of $m$ it will eventually reach $m$ (as $a$ is transient to the right), and the fact that the path it will make on $[m,\infty)$ depends only on the restriction of $a$ to $[m,\infty)$.
%\end{proof}

We may now give a meaning to taking the initial positions to $-\infty$. For $m\in\Z$ denote by $m^{(k)}$ the constant sequence $x_1,...,x_k$ with values $m$.
The main observation is the following:
\begin{lemma}\label{le:MovingLeftDoesntMatter}
Let $a$ be a strongly $k$-right transient arrow environment. $LO^{x_1,\ldots,x_k}(a)(x,i) = LO^{m^{(k)}}(x,i)$ for all $x\geq m$ and $i\geq 1$, whenever $x_1,\ldots,x_k<n$.
That is, the $k$-leftover environment to the right of $n$ remains the same for any choice of starting points to the left of $n$.
\end{lemma}
\begin{proof}
Fix $x_1,...,x_k$. Every time any of the walker is to the left of $m$ it will eventually reach $m$ (as $a$ is $k$-right transient w.r.t.\ $x_1,...,x_k$). But the asymptotic local time of the $k$-mob walk on $[m,\infty)$ depends only on the restriction of $a$ to $[m,\infty)$, which concludes the proof.
\end{proof}

This allows us to define the leftover environment left by $k$ walkers ``starting from $-\infty$"
\begin{defi}
Given a $k$-right transient arrow environment $a$, the stationary $k$-leftover environment $LO^{k\text{-limit}}(a)$ is defined by $LO^{k\text{-limit}}(a)(x,i) = LO^{x^{(k)}}(a)(x,i)$ for all $x\in \Z$ and $i\geq 1$.
\end{defi}

Let $P$ be a measure over cookie environments, and recall that $\PP$ is the measure over arrow environments associated to $P$. We denote by $\plo^{x_1,\dots,x_k}$ the pushforward measure obtained from $\PP$ by the map $\varphi^{x_1,\ldots,x_k}:a\rightarrow LO^{x_1,\ldots,x_k}(a)$, and by $\plo^{k\text{-limit}}$ the pushforward measure obtained from $\PP$ by the map $\varphi^{k\text{-limit}}: a\rightarrow LO^{k\text{-limit}}(a)$.

%Let $\theta$ be the left shift. Note that for any environment $a$ and any $m\in \Z$ $LO^0(\theta^m(a)) = \theta^m(LO^m(a))$.
\begin{lemma}\label{lem:stationarizationIsStationary}
Let $P$ be (SE), (ND), (WEL), or (ELL) probability measure over cookie environments. If $\PP$ is $k$-right transient, then $\plo^{k\text{-limit}}$ is also (SE), (ND), (WEL), or (ELL) respectively.
\end{lemma}
\begin{proof}

(ND), (WEL), and (ELL) are straightforward. Assume that $P$ satisfies (SE). To show stationarity, it is enough to show that $\plo^{k\text{-limit}}(A) = \plo^{k\text{-limit}}(\theta A)$ for any event $A$ depending only on the arrows above finitely many positions, where $\theta$ is the left shift. Assume that $A$ depends only on the arrows $\{a(x,\cdot), x\in [-m,m]\}$. Then for $n<-m$ we have
$$ \plo^{k\text{-limit}}(A)  = \plo^{(-n)^{(k)}}(A) = \plo^{(-n+1)^{(k)}}(\theta A) = \plo^{k\text{-limit}}(\theta A) $$
Where the first and last equalities follows from Lemma \ref{le:MovingLeftDoesntMatter}, and the middle equality follows from stationarity of $\PP$.
Ergodicity will follow from stationarity once we show the map $\varphi^{k\text{-limit}}:a\rightarrow \plo^{k\text{-limit}}$ is measurable, as any factor of an ergodic system is ergodic (see e.g.\ Lemma 2.1 of \cite{amir2013zero}).
To show measurability of this map, we first note that the maps $\varphi^{x_1,\ldots,x_k}$ are measurable. Indeed, by $k$-right transience, they are a.s.\ the pointwise limit of the functions $\varphi_t^{x_1,\ldots,x_k}(a)(x,i) := a(x,i+L_t^{(k-\min)}(x))$, the environment leftover after $t$ steps of the min-$k$ walk on $a$ with initial positions $x_1,\ldots,x_k$. The latter are actually continuous, as they are the composition of $t$ single-step functions.
To show that $\varphi^{k\text{-limit}}$ is measurable it is enough to show that for any $r\in \N$ and any event $A$ depending only on the arrows $\{a(x,\cdot),\ x\in [-r,r]\}$, there exists a measurable set $C$ s.t. $\PP(B\triangle C)=0$, where $B$ is the inverse image of $A$, that is $B:=\big(\varphi^{k\text{-limit}}\big)^{-1}(A)$. Choose $C:=\big(\varphi^{(-r)^{(k)}}\big)^{-1}(A)$, then measurability of $\varphi^{(-r)^{(k)}}$ implies that $C$ is measurable, and by Lemma \ref{lem:stationarizationIsStationary} we have $\PP(B\triangle C)=0$ as $\varphi^{k\text{-limit}}$ and $\varphi^{(-r)^{(k)}}$ agree on $A$ for any $k$-right transient arrow environment.
 \end{proof}

\begin{prop}
Let $P$ be a (SE) and (ND) measure over cookie environments. The following are equivalent:
\begin{enumerate}
\item $P$ is strongly $(k+1)$-right transient.\\
\item For some $x_1,\ldots,x_k\in \Z$, $\plo^{x_1,\dots,x_k}$ is transient to the right. \\
\item $\plo^{k\text{-limit}}$ is transient to the right
\end{enumerate}
More so, when the above clauses hold, the speed of the walker on $LO^{x_1,\dots,x_k}(a)$ and $LO^{k\text{-limit}}(a)$ is a.s.\ equal.
\end{prop}

\begin{proof}
We may assume $P$ is $k$-right transient, otherwise all clauses fail trivially. The equivalence of (1) and (2) follows directly from Corollary \ref{cor:strongtransience}. To see that (1) implies (3), consider a $(k+1)$-walk with all $k+1$ walkers starting at $0$. Corollary \ref{cor:T.infinite.Vs.X.transient} and Lemma \ref{le:inf_opt_reg} imply that there are a.s.\ a positive density of regeneration positions, i.e. infinitely many positions $x>0$ for which the directed edge $(x,x-1)$ is never crossed. Therefore for any $\epsilon>0$ there is some $m>0$ such that with probability $\geq 1-\epsilon$ there is some regeneration position in the interval $[0,m]$. Consider now a $(k+1)$-mob walk $X$ with all walkers starting at $-m$. Stationarity of $P$ ensures that with probability $\geq 1-\epsilon$ there exists a position $x\in [-m,0]$ for which the local time is $k+1$. Sample an arrow environment $a$ from the induced measure $\PP$, which we may assume to be $(k+1)$-right transient. Last, consider a walk $Y$ on $a$ where $k$ particles start at $-m$ and the last particle starts at $0$. Lemma \ref{le:initial positions} tells us that $L^Y(x)\leq L^X(x)$, which is equivalent to the statement that the number of times each directed edge is crossed in $Y$ is less or equal to the number of times it is crossed in $X$ (as the number of crossings of the edge $(x,x\pm 1)$ is simply the number of right/left arrows in the first $L^X(x)$ arrows above $x$). In particular if $x\in [-m,0]$ is a regeneration position then the directed edge $(x,x-1)$ is never crossed in $Y$, and therefore the walker starting at $0$ will never reach $x-1$. Let $A_m$ be the event that a walker starting at $0$ will go to $+\infty$ without reaching $-m-1$. Then the above discussion shows that for any $\epsilon$ there is an $m$ s.t. $\plo^{(-m)^{(k)}}(A_m) > 1-\epsilon$. As this event depends only on the environment above $[-m,\infty)$, Lemma \ref{le:MovingLeftDoesntMatter} gives $\plo^{k\text{-limit}}(A_m)=\plo^{(-m)^{(k)}}(A_m) \geq 1-\epsilon$. As $\epsilon$ was arbitrary, this gives that $\plo^{k\text{-limit}}$ is transient to the right.

(In fact, once we know that the probability of going to $\infty$ is positive, this also follows from the 0-1 law for directional transience - Theorem \ref{thm:ABO}).
To get that (3) implies (2), note that right-transience of $\plo^{k\text{-limit}}$ is equivalent to $\plo^{k\text{-limit}}(A_m)\rightarrow 1$ as $m\rightarrow \infty$. Taking $m$ large enough so that $\plo^{k\text{-limit}}(A_m)>\frac12$, and using again Lemma \ref{le:MovingLeftDoesntMatter} we get $\plo^{(-m)^{(k)}}(A_m) >\frac12$. By Lemma \ref{le:initial positions} moving the starting points to $0$ cannot increase local time, and therefore $\PP(a \text{ is $(k+1)$-right transient w.r.t. } 0^{(k+1)})>\frac12$. Theorem \ref{thm:TransienceThreshold} now gives that $a$ is $(k+1)$-right transient w.r.t.\ $0^{(k+1)}$ $P$-a.s.\

Last, to get the statement on speeds we argue that conditioned on $A_m$ the speed depends only on the environment above $[-m,\infty)$, which is identical under $LO^{x_1,\dots,x_k}(a)$ and $LO^{k\text{-limit}}(a)$ for any $(k+1)$-right transient $a$. Since when (1)-(3) hold, with probability $1$ $A_m$ holds for some $m$, the result follows.
\end{proof}

The last Lemma allows us to transfer the results in the other sections of this paper to (SE) leftover environments.
We sum this up in the next corollary.
Given a measure $P$ over cookie environments, let $X_{se}^{(k)}$ denote the excited random walk on an environment sampled according to $\plo^{(k-1)\text{-limit}}$.

\begin{cor}
Theorems \ref{thm:TransienceThreshold}, \ref{thm:bddCaseTransience}, \ref{thm:ZernerLLN} and \ref{thm:bddCaseSpeed} hold with $X^{(k)}$ replaced by $X_{se}^{(k)}$.
More so, the values of $R$ and $v_k$ are invariant under this substitution. In particular, for probability measures over cookie environments which are (IID), (BD), (WEL) and with $\delta>3$, the walker $X^{(2)}$ on the leftover environment is an example of ERW is a stationary ergodic environment with positive speed, as promised in the introduction.
\end{cor}

\section{Concluding remarks and open problems}\label{sec:Remarks}

\subsection{Remarks}
\begin{enumerate}
\item The main purpose of this paper is to introduce the $k$-particle picture and some technique of dealing with it.
 There is an extensive research in the field for the case of one walker in environments which satisfy (BD) and (IID), for example Kosygina-Zerner \cite{kosygina2008positively} and \cite{kosygina2012excited} (transience versus recurrence, ballisticity, CLT), Basdevant-Singh \cite{basdevant2008speed} and \cite{basdevant2008rate} (ballisticity and asymptotic rate of diffusivity), Peterson \cite{peterson2012large} and \cite{peterson2012strict} (law of large deviation, slow-down phenomenon, and strict monotonicity results), Rastegar-Roitershtein \cite{rastegar2011maximum} (maximum occupation time) and Dolgopyat-Kosygina \cite{dolgopyat2011central}, Kosygina-Mountford \cite{kosygina2011limit} and Kosygina-Zerner \cite{kosygina2013excursions} (limit laws).
 In this paper we focused on generalizing results regarding transience vs.\ recurrence and positive speed for one walker on such environments to $k$ excited walkers. We believe that by pushing the proofs of other results through the machinery described in this paper, many other results could be generalized to the $k$-particle picture.
\item The arguments in Sections $2$ and $3$ connect the transience of the process $Z^+$ and the right transience of the walkers on the leftover environments under the assumptions (SE) and (ELL). Thus if one is able to give criterions for transience of $Z^+$, one gets criterions for transience for walks on the leftover environments. When the environment is (IID) and (BD) there is an exact criterion for transience of $Z^+$ in terms of $\delta$ by viewing $Z^+$ as a branching processes with migration (see Chapter 5). In a recent work \cite{Kosh}, criterions for transience of $Z^+$ are given for more general environments, such as periodic environments, where the parameter $\delta$ is replaced by a more robust parameter $\theta$ which coincides with $\delta$ when the environment is either (BD) or (POS). Thus Theorem \ref{thm:bddCaseTransience} can be generalized to such periodic environments, giving that $X^{(j)}$, $j\le k$, are all transient to the right if and only if $\theta>k$.
\end{enumerate}
\subsection{Open problems}

For simplicity of presentation we shall assume in this section that $P$ is a probability measure over the space of cookie environments and that whenever $\delta = E \left[ \sum_{i=1}^\infty(2\om(0,i)-1)\right ]$ is defined then it is in $[0,\infty]$.
Our first two open problem deal with removing the boundedness condition on the cookies from Theorems \ref{thm:bddCaseTransience} and \ref{thm:bddCaseSpeed}. As we ask for the same threshold, we still demand that $\delta$ be well-defined:
\begin{problem}
Assume that $P$ satisfies (IID) and (WEL) and that $\delta$ is well-defined. Is it true that $X^{(k)}$ is transient if and only if $k<\delta$.
\end{problem}
\begin{problem}
Assume that $P$ satisfies (IID) and (WEL) and that $\delta$ is well-defined. Is it true that $X^{(k)}$ has non-zero speed if and only if $k<\delta-1$.
\end{problem}
For Theorem \ref{thm:bddCaseTransience} one may also try to generalize conditions by relaxing the (IID) condition to (SE):
\begin{problem}
Assume that $P$ satisfies (SE), (WEL) and (BD). Is it true that $X^{(k)}$ is transient if and only if $k<\delta$.
\end{problem}
Note that the $1$ particle version of this question appeared as problem $3.11$ of \cite{kosygina2012excited}.

As mentioned after Theorem \ref{thm:bddCaseSpeed}, in the (IID), (BD), and (WEL) case, when $k + 1 < \delta \le k + 2$ the first $k-1$ walkers are transient to the right with positive speed, the next one is transient to the right with $0$ speed, and all subsequent walkers are recurrent. We now ask two questions regarding how general is this phenomenon. Note that we do not ask for the values of the thresholds, which allows us more freedom in the conditions on $\mu$:
\begin{problem}
Is it true that for any $P$ satisfying (SE) and (ELL) there exists some $R\geq 0$ so that a.s.\ the first $R$ walkers are transient with non-zero speed, the next walker is a.s.\ transient with $0$ speed, and all following walkers are a.s.\ recurrent?
\end{problem}
And the weaker form:
\begin{problem}
Is it true that for any $P$ satisfying (SE) and (ELL) if $v_k=0$ then $v_m=0$ for all $m\geq k$?
\end{problem}

We end with a problem of a different flavour. Say that we are given 2 walkers on an environment, and we are allowed to decide a which walker to move at each step. How much can we control the path of one of the walkers? For 2-transient arrow environments, the Exchangeability Lemma \ref{lem:exchangeabilty} tells us that the local time is invariant under the choice of scheduling, and thus we have very limited control, while for recurrent arrow environments one can completely control the path of one walker by letting him walk only on arrows compatible with the path and letting the other walker clear out all the ``wrong" arrows it encounters.
When given walkers on a cookie environment, one cannot hope to attain a pre-described path as the moves are random, so instead we ask only to have one of the walkers go a.s.\ to infinity.
It is clear that for some environments this is impossible, for instance if all cookies are placebo $(\frac12,\frac12)$ cookies. On the other hand, as observed by Jonathon Peterson (private communication), when there are infinite cookies per site with infinite drift in both directions, one may imitate the arrow strategy and have one walker eat only positive cookies. Our question is therefore the following:

\begin{problem}[Master and Servant]
Is there a $1$-recurrent non-degenerate bounded cookie environment $\om$ and a $2$-scheduling so that one of the two particles will a.s.\ go to infinity?
\end{problem}

\section*{Acknowledgments}
We thank Noam~Berger and Gady~Kozma for useful discussions. The research of G.A.\ was supported by the Israel Science Foundation grant ISF 1471/11.
The research of T.O\ was partly supported by the Israel Science Foundation.
\bibliography{mob}{}
\bibliographystyle{plain}

\end{document}